\documentclass[a4paper,notitlepage]{amsart}

\usepackage[latin1]{inputenc}
\usepackage[T1]{fontenc}
\usepackage{amsmath}
\usepackage{bbm}
\usepackage{amsfonts}
\usepackage{amssymb}
\usepackage{latexsym}
\usepackage[english]{babel}
\usepackage[dvips]{graphicx}
\usepackage{longtable}
\usepackage{amsthm}
\usepackage{hyperref}
\usepackage{dsfont}
\usepackage{geometry}
\usepackage{color}
\usepackage{mathrsfs}
\usepackage{esint}
\usepackage{enumitem}

\newcommand{\dx}{\, \text{\textnormal{d}}}
\newcommand{\la}{\langle}

\newcommand{\R}{\mathbb{R}}

\newcommand{\eps}{\varepsilon}
\newcommand{\leb}{\mathcal{L}}

\renewcommand{\phi}{\varphi}

\newcommand{\ind}{\mathbbm{1}}
\renewcommand{\bar}{\overline}
\renewcommand{\tilde}{\widetilde}
\newcommand{\Lip}{\operatorname{Lip}}

\newcommand{\supp}{\operatorname{supp}}

\newcommand{\id}{\text{id}}
\newtheorem{theorem}{Theorem}[section]
\newtheorem{proposition}{Proposition}[section]

\newtheorem{definition}{Definition}[section]
\newtheorem{lemma}{Lemma}[section]
\DeclareMathOperator*{\smallo}{o}
\newcommand{\dint}{\displaystyle\int}
\newcommand{\dsum}{\displaystyle\sum}

\newcommand{\Ent}{\operatorname{Ent}}

\renewcommand{\O}{\mathcal{O}}
\newcommand{\sm}{\text{\normalfont small}}
\renewcommand{\la}{\text{\normalfont large}}

\renewcommand{\O}{\mathcal{O}}

\newcommand{\e}{\text{\textnormal e}}

\newcommand{\res}{\mathop{\hbox{\vrule height 7pt width .5pt depth 0pt
\vrule height .5pt width 6pt depth 0pt}}\nolimits}

\begin{document}

\title[Entropic regularization of the Monge problem on the real line]{The entropic regularization of the Monge problem \\ on the real line}
\author{Simone Di Marino}
\address{Indam, Scuola Normale Superiore, Piazza dei Cavalieri 7, 56126 Pisa PI, Italy}
\email{simone.dimarino@altamatematica.it}
\author{Jean Louet}
\address{Ceremade, CNRS, UMR 7534, Universit\'e Paris-Dauphine, PSL Research University, 75016 Paris, France}
\address{Inria-Paris, Mokaplan, 2 rue Simone Iff, 75012 Paris, France}
\email{louet@ceremade.dauphine.fr}

\begin{abstract}
    We study the entropic regularization of the optimal transport problem in dimension~1 when the cost function is the distance $c(x,y)=|y-x|$. The selected plan at the limit is, among those which are optimal for the nonpenalized problem, the most ``diffuse'' one on the zones where it may have a density.
\end{abstract}

\maketitle

\section{Introduction}

In this paper, we are concerned with the following approximation of the optimal transportation problem: given two probability measures~$\mu,\nu$ on $\R^d$ and $\eps>0$, find the minimizer of
\begin{equation} \label{EntRegIntro} J_\eps : \gamma \mapsto \iint |y-x|\dx\gamma(x,y) + \eps \iint \log \dfrac{\text{d}\gamma}{\text{d}(\mu\otimes\nu)} \dx\gamma  \end{equation}
among all the measures $\gamma$ on $\R^d\times\R^d$ which have a density with respect to $\mu\otimes\nu$ and whose first and second marginals are equal to $\mu$ and $\nu$, respectively. At the limit $\eps\to 0$, we expect the minimizer $\gamma_\eps$ to converge to an optimal measure for the energy $\int|y-x|\dx\gamma$ with prescribed marginals, and our goal is to understand which one is selected. \medskip

The corresponding problem with $\eps=0$ is the ``distance cost'' version of the classical optimal transport problem, whose original formulation, due to Monge in the 18th~century~\cite{Monge}, consists in looking for the map $T:\R^d\to\R^d$ which minimizes
$$ \int_{\R^d} c(x,T(x))\dx\mu(x), $$
where $c$ is a given positive function and $T$ must satisfy the constraint
\begin{equation} \text{for any Borel set } B\subset \R^d,\; \mu(T^{-1}(B)) = \nu(B). \label{immes} \end{equation}
Due to the difficulty of this image-measure constraint, this problem remained quite difficult to solve for many years. A suitable relaxation, which corresponds to our problem~\eqref{EntRegIntro} with no penalization, was introduced by Kantorovich in the 1940s~\cite{Kan1,Kan2}, namely, the minimization problem:
\begin{equation} \inf\left\{ \iint_{\R^d\times\R^d} c(x,y) \dx\gamma(x,y) \,:\, \gamma\in \Pi(\mu,\nu) \right\}, \label{kant} \end{equation}
where $\Pi(\mu,\nu)$ is the set of probability measures $\gamma$ on $\R^d\times\R^d$ having $\mu$ and $\nu$ as marginals, that is,
\begin{equation} \label{marginals} \text{for any Borel set } B \subset \R^d, \qquad \gamma(B\times\R^d) = \mu(B) \quad\text{and}\quad \gamma(\R^d\times B) = \nu(B). \end{equation}
The problem~\eqref{kant} is a generalization of the Monge's problem as, from any map $T$ satisfying the above measure constraint, it is easy to build a measure $\gamma \in \Pi(\mu,\nu)$ which is concentrated on the graph of~$T$ and has same total energy; moreover, due to the compactness properties of $\Pi(\mu,\nu)$ for the weak topology of measures, proving the existence of solutions of~\eqref{kant} is easy by the direct methods of the calculus of variations. Thanks to a suitable convex duality argument, it is possible to prove~\cite{BreCRAS,BrePolar,GanMcc}, under suitable assumptions on the data and the cost function $c$ (which include the case $c = |y-x|^p$ for $1 < p <+\infty$), that the problem~\eqref{kant} admits a unique solution which~is induced by a map~$T$, yielding optimality of this map for the original Monge problem. In the case $c=|y-x|$, the existence results are more recent and the uniqueness is not guaranteed anymore; cf.~\cite{AmbLN}. We refer to the monographs~\cite{AGS,Vi1,Vi2,SanOTAM} for a complete overview of the optimal transportation theory.

Although it was already present in much earlier works, the penalization~\eqref{EntRegIntro} has been recently reintroduced for numerical reasons. Indeed, computing numerically the optimal transport map~$T$ remained for a long time a very challenging problem; a~first major achievement appeared in the beginning of the 2000s, when Benamou and Brenier~\cite{BenBre} introduced the so-called ``dynamical formulation'' (based on the minimization of the kinetic energy among the curves of measures and velocity fields satisfying a mass conservation equation), which can be solved by an augmented Lagragian method after a convex change of variables. Let us also mention the algorithms due to Angenent, Hacker, and Tannenbaum~\cite{AngHakTan}, which also rely on fluid-mechanics formulations, and more recently the approaches by discretization of the Monge--Amp\`ere equation~\cite{BenCarMerOud, bcm16} or via semidiscrete optimal transport~\cite{merigot2011,levy2015}.  \smallskip

The numerical interest of the approaches similar to~\eqref{EntRegIntro} has been shown in the last few years. The general idea consists in perturbing the Kantorovich problem by the so-called ``entropy functional'' defined as
$$ \Ent(\gamma|\rho) = \left\{\begin{array}{ll} \dint \log\left(\dfrac{\text{d}\gamma}{\text{d}\rho}\right) \dx\gamma & \text{if } \gamma \ll \rho, \\[2mm] +\infty & \text{otherwise.} \end{array} \right. $$
and in focusing, given a suitable small parameter $\eps>0$, on the problem
\begin{equation} \inf\left\{ \iint c(x,y)\dx\gamma(x,y) + \eps \Ent(\gamma|\mu\otimes\nu) \,:\, \gamma\in\Pi(\mu,\nu)\right\}. \label{entrkant} \end{equation}
Among other properties, the function $\Ent(\cdot|\mu\otimes\nu)$ enforces $\gamma$ to be an absolutely continuous measure with respect to the tensor product $\mu\otimes\nu$ and favors such transport plans which are {\it as diffuse as possible} (on the set $\Pi(\mu,\nu)$, its unique minimizer is $\mu\otimes\nu$). The idea of the entropic penalization actually goes back to Schr\"odinger's works~\cite{Schr}; moreover, the algebraic properties of the entropy functional and of the dual problem of~\eqref{entrkant} make the numerical computation of its solution much easier, thanks to the so-called Sinkhorn's algorithm involving alternated projections. We refer to the papers~\cite{Cut,BenCarCutNenPey} and the Ph.D.~thesis~\cite{NenPHD} for more details on the theoretical and numerical properties of this class of algorithms. \smallskip

From a theoretical point of view, the convergence of the solution of~\eqref{entrkant} as $\eps\to 0$ has been recently proven in~\cite{CarDuvPeySch} (we also mention~\cite{Leo1,Leo2}, in which similar problems are studied  in a much more abstract framework): therein, the authors showed that the family of functionals
\begin{equation} \gamma \in \Pi(\mu,\nu) \mapsto \int c\dx\gamma + \eps \Ent(\gamma|\mu\otimes\nu) \label{energieeps} \end{equation}
is $\Gamma$-converging, as $\eps\to 0$, to the transport energy $\gamma\mapsto \int c\dx\gamma$. Therefore, when $c$ is one of the costs for which there exists a unique optimal plan $\gamma$ for the Kantorovich problem~\eqref{kant}, the family $(\gamma_\eps)_\eps$ of minimizers of~\eqref{energieeps} converges to~$\gamma$. \medskip

In this paper, we are interested in the following theoretical question: {\it What's happening if there are several minimal plans $\gamma$ for the cost function~$c$?} As usual in minimization problems of penalized functionals, it is natural to guess that, when the set~$\mathcal{O}_c$ of optimal plans for cost function~$c$ has at least two elements, the family~$(\gamma_\eps)_\eps$ of minimizers of $J_\eps$ will converge to the ``most diffuse'' element of $\mathcal{O}_c$, namely, the one which minimizes the entropy functional $E(\cdot|\mu\otimes\nu)$. However, the existence of a plan belonging to~$\mathcal{O}_C$ and having {\it finite entropy} is far from clear, and is generally false depending on the cost function~$c$ and the data.

We here focus on the case where $c(x,y) = |y-x|$, which was the original cost function proposed by Monge. This cost function is certainly the most studied one for which it is known that the uniqueness of optimal plan and optimal map fails, and selecting a particular optimal plan by adding a ``regularizing'' term to the energy is quite common: this is classically done through strictly convex costs which brings to the {\it monotone transport}; see, for instance, \cite{CafFelMcC,FelMcC,TruWan} for existence of optimal maps and \cite{LiSanWan} for very partial regularity results (we also mention the quite different approximation proposed in~\cite{DepLouSan}, where the ``regularizing'' term enforces the transport plan to be induced by a regular map). In the present paper, we concentrate on the entropic regularization of this problem where the measures are supported on the real line. In higher dimension, let us just notice that the structure of optimal Monge's plans is much more complicated, involving the geometric notion of {\it transport rays} (see, for instance,~\cite{AmbLN} or \cite[Chapter~3]{SanOTAM}), and although one may guess that the ``most diffuse transport plans on each ray'' will be selected, attacking this regularization problem on the Euclidean space would probably require very different techniques from ours. \medskip

In the one-dimensional case, the contributions and results of this paper are the following:
\begin{itemize}
\item First of all, we need a complete description of the set of one-dimensional optimal transport plans for the distance cost $c(x,y)=|y-x|$; although this result is natural and its proof is not complicated, it was not present in the literature at that time to the best of our knowledge. The structure of optimal plans is described in Proposition~\ref{1DOptPlans} and can be summarized as follows: denoting by~$T$ the monotone rearrangement (that is, the unique non-decreasing transport map) between~$\mu$ to~$\nu$, the optimal plans~$\gamma$ are those such that if $(x,y)$ is in the support of $\gamma$, then
\begin{itemize}
\item if $T(x)=x$, then $y$ must be exactly equal to $x$;
\item if $x$ belongs to some interval where $T-\id >0$, then $y$ must belong to the same interval and be larger than $x$;
\item the analogous constraint holds where $T-\id<0$.
\end{itemize}
In particular, any optimal transport plan can be decomposed into a ``bad plan'' which is concentrated on the graph on the identity map (and therefore is singular with respect to the Lebesgue measure) and a ``possibly good plan'' which must only respect the sign of $T-\id$ and is allowed to have a density.
\item Once this structural result is stated, our main result studies the $\Gamma$-convergence of the penalized functional. Actually, we need to consider the ``rescaled functional''
$$ F_\eps : \gamma \mapsto \frac{1}{\eps} \left(\int |y-x|\dx\gamma-W_1(\mu,\nu)\right) + \Ent(\gamma|\mu\otimes\nu)-\mu(A)|\log(2\eps)|, $$
where $A$ is the set $\{ x\in\R : T(x)=x \}$ and $W_1(\mu,\nu)$ the minimal transport energy for the distance cost. For fixed $\eps$, the functional $F_\eps$ admits exactly the same minimizers as~$J_\eps$. In Theorem~\ref{maintheo} we prove, under technical assumptions on the data, that its $\Gamma$-limit is $+\infty$ outside of the set of optimal plans, and that, in this set, it is equal to
$$ F: \gamma \mapsto \Ent(\gamma\res(\R\setminus A)^2 | \mu\otimes\nu) + C, $$
where the constant $C$ only depends on the data. In particular, the unique minimizer of $F$ is the plan which is optimal for distance cost and whose ``good part'' is ``as diffuse as possible'', as we could naturally expect. Moreover, this result implies the asymptotic expansion
$$ \min J_\eps = W_1(\mu,\nu) + \mu(A)\eps|\log(2\eps)| + \eps \min F + \smallo(\eps); $$
notice that the excess of order $\eps|\log\eps|$ is a common phenomenon with the penalization of the Monge problem proposed in~\cite{DepLouSan}.
\item Finally, Section~\ref{LastSection} studies the explicit form of the optimal plan which is selected at the limit, i.e.,~the minimizer of~$F$; the precise result is stated in Theorem~\ref{maintheo2}, which also gives a necessary and sufficient condition for $F$ to be not identically $+\infty$. This section relies on entropy-minimization problems of the ``good part of optimal plans'', which can be rewritten
  $$ \min\Big\{ \Ent\big(\gamma\,|\, (\mu(x)\otimes\nu(y))\res(I\times I) \cdot \ind_{y\geq x}\Big\} $$
  on maximal positivity intervals $I$ of $T-\id$ (and as analogous minimization problems on their negative counterparts). Similar problems were already deeply analyzed in~\cite{Borwein1,Borwein2}, and we here provide new self-contained proofs more suitable for our needs.
\end{itemize}

\section{Notation and preliminary results}

In this section, we collect the notation and well-known facts on measure theory and optimal transportation which will be used throughout the paper.

Let $X$ and $Y$ be two Polish spaces (in this paper, we will only focus on the case $X=Y=\R$), and let $\mu$, $\nu$ be two positive measures on $X$, $Y$ whose total masses are finite and equal. We denote by $\Pi(\mu,\nu)$ the set of {\it transport plans} from $\mu$ to $\nu$, that is, the set of positive measures on $X\times Y$ satisfying~\eqref{marginals}; recall that this constraint can be reformulated as follows: for any $(\phi,\psi) \in C_b(X) \times C_b(Y)$,
$$ \iint \phi(x)\dx\gamma(x,y) = \int\phi\dx\mu \quad \text{and}\quad \iint\psi(y)\dx\gamma(x,y) = \int\psi\dx\nu. $$
When $T:X\to Y$ is a map, we call it a {\it transport map} from $\mu$ to $\nu$ if it satisfies~\eqref{immes}. This is equivalent to saying that the measure $\gamma_T$ defined as
$$ \text{for any } f \in C_b(X\times Y),\; \iint f\dx\gamma_T = \int f(x,T(x))\dx\mu(x) $$
belongs to $\Pi(\mu,\nu)$. Equivalently, this means that the equality $\int \phi\circ T\dx\mu = \int \phi\dx\nu$ holds for any $\phi \in C_b(X)$. A useful property of $\Pi(\mu,\nu)$ is the following.

\begin{proposition} Let $X,Y$ be Polish spaces, and let $\mu,\nu$ be positive measures on $X,Y$ with finite mass. Then the set $\Pi(\mu,\nu)$ is compact with respect to the narrow topology of measures. \end{proposition}

We now introduce the definition of {\it relative entropy}.

\begin{definition} Let $\rho$ be a fixed positive measure on $\R^d$, $d\geq 1$. For any positive measure $\gamma$ on $\R^d$, we set
$$ \Ent(\gamma|\rho) := \left\{\begin{array}{ll} \dint \log\left(\dfrac{\text{\normalfont d}\gamma}{\text{\normalfont d}\rho}\right) \dx\gamma & \text{if } \gamma \ll \rho, \\[2mm] +\infty & \text{otherwise,} \end{array} \right.  $$
  and the functional $\Ent(\cdot|\rho)$ is called {\em relative entropy with respect to $\rho$.} \end{definition}

The following properties of the functional $\Ent$ are classical (see, for instance, \cite[Theorem~2.34]{AmbFusPal}).

\begin{proposition} Let $\rho$ be a fixed positive measure on $\R^d$, $d\geq 1$. Then the functional $\Ent(\cdot|\rho)$ is strictly convex, and it is lower semicontinuous with respect to the narrow convergence of measures. \end{proposition}

In this paper, we are interested in the optimal transport problem when $c$ is given by the distance $c(x,y)=|y-x|$. In that case, the so-called ``duality formula of the optimal transport problem'' takes the following form.

\begin{theorem} \label{theodual} Given $\mu,\nu \in \mathcal{P}(\R^d)$, we have the equality
\begin{equation} \label{PbDual} \inf\left\{ \int |y-x|\dx\gamma(x,y) \,:\, \gamma\in\Pi(\mu,\nu) \right\} = \sup\left\{ \int u\dx\nu-\int u\dx\mu \,:\, u \in \Lip_1(\R^d) \right\}, \end{equation}
where $\Lip_1(\R^d)$ stands for the set of Lipschitz functions on $\R^d$ having Lipschitz constant at most~$1$. The common optimal value will be denoted by $W_1(\mu,\nu)$ and the set of optimal transport plans by $\O_1(\mu,\nu)$; a maximizer of the dual problem will be called a {\em Kantorovich potential}.

As a consequence of \eqref{PbDual}, given a function $u\in\Lip_1(\R^d)$ and a transport plan $\gamma\in\Pi(\mu,\nu)$, the following properties are equivalent:
\begin{enumerate}[label=(\roman*)]
\item $u$ is a Kantorovich potential and $\gamma\in\O_1(\mu,\nu)$;
\item for $\gamma$-a.e.~$(x,y)$, $|y-x| = u(y)-u(x)$.
\end{enumerate}
\end{theorem}




The last theoretical notion we will need in this paper is the one of $\Gamma$-convergence.

\begin{definition} \label{defgammalim} Let $X$ be a complete metric space and $(F_n)_n, F$ be functionals $X\to \R\cup\{+\infty\}$. We say that $(F_n)_n$ $\Gamma$-converges to $F$ if, for any $x \in X$, the two following inequalities are satisfied:
\begin{itemize}
\item for any sequence $(x_n)_n$ of $X$ converging to $x$, $ \liminf\limits_n F_n(x_n) \geq F(x)$ {\em ($\Gamma$-liminf inequality)};
\item there exists a sequence $(x_n)_n$ of $X$ converging to $x$ such that $ \limsup\limits_n F_n(x_n) \leq F(x)$ {\em ($\Gamma$-limsup inequality)}.
\end{itemize} \end{definition}

The main interests of this notion are its implications in terms of convergence of minima and minimizers.

\begin{theorem} \label{gammamin} Let $(F_n)_n$ be a sequence of functions $X \to \R\cup\{\pm\infty\}$ and assume that $F_n \xrightarrow[n]{\Gamma} F$. Assume moreover that there exists a compact and nonempty subset $K$ of $X$ such that
$$ \forall n\in N, \; \inf_X F_n = \inf_K F_n $$
(we say that $(F_n)_n$ is equi-mildly coercive on $X$). Then $F$ admits a minimum on $X$ and the sequence $(\inf_X F_n)_n$ converges to $\min F$. Moreover, if $(x_n)_n$ is a sequence of $X$ such that
$$ \lim_n F_n(x_n) = \lim_n (\inf_X F_n)  $$
  and if $(x_{\phi(n)})_n$ is a subsequence of $(x_n)_n$ having a limit $x$, then $ F(x) = \inf_X F $. \end{theorem}

In our settings, all the functionals that we consider will be defined on the space $\Pi(\mu,\nu)$. Thanks to its compactness property, the equi-coercivity assumption of Theorem~\ref{gammamin} will always be satisfied in this paper: therefore, in order to conclude that the minimal values and the minimizers of a given family of functionals are converging to those of a given ``target'' functional, checking the upper and lower limit conditions of Definition~\ref{defgammalim} will be enough. \smallskip

We recall now the already known $\Gamma$-convergence results in entropic regularization (see~\cite[Theorem~2.7]{CarDuvPeySch} where the proof is given for $c(x,y)=|y-x|^2$ but can easily be adapted for a much larger class of cost functions).

\begin{theorem}[zeroth-order $\Gamma$-convergence] \label{ZeroGammaCV} Let $c:\R^d\to\R$ be continuous and $\mu,\nu$ be two probability measures on $\R^d$ having compact support and being absolutely continuous with respect to the Lebesgue measure. Assume that
$$ \int \log\frac{\text{\normalfont d}\mu}{\text{\normalfont d}\leb^d} \dx\mu,\,\int \log\frac{\text{\normalfont d}\nu}{\text{\normalfont d}\leb^d} \dx\nu < +\infty.  $$
Then the family of functionals
$$ J_\eps : \gamma \in\Pi(\mu,\nu) \mapsto \int c\dx\gamma + \eps \Ent(\gamma|\mu\otimes\nu) $$
is $\Gamma$-converging, as $\eps \to 0$, to $\gamma \mapsto \dint c\dx\gamma$. \end{theorem}

We end this section by the ``first-order $\Gamma$-convergence'' result, which completes the previous one.

\begin{proposition}[first-order $\Gamma$-convergence] \label{FirstGammaCV} Under the same assumptions as Theorem~\ref{ZeroGammaCV}, denoting by $W_c(\mu,\nu)$ the minimal transport energy for cost $c$ and by $\mathcal{O}_c(\mu,\nu)$ the set of corresponding optimal plans, the family of functionals
$$ H_\eps : \gamma \in \Pi(\mu,\nu) \mapsto \frac{1}{\eps}\left(\int c\dx\gamma - W_c(\mu,\nu)\right) + \Ent(\gamma|\mu\otimes\nu) $$
is $\Gamma$-converging to
$$ H : \gamma  \in \Pi(\mu,\nu) \mapsto \left\{ \begin{array}{ll} \Ent(\gamma|\mu\otimes\nu) & \text{if } \gamma \in \mathcal{O}_c(\mu,\nu), \\ +\infty &\text{otherwise}. \end{array} \right. $$
\end{proposition}

\begin{proof} The $\Gamma$-limsup inequality is trivial since, given $\gamma \in \Pi(\mu,\nu)$,
\begin{itemize}
\item if $H(\gamma) = +\infty$, there is nothing to prove;
\item if $H(\gamma) < +\infty$, this means that $E(\gamma|\mu\otimes\nu)$ is finite and $\int c\dx\gamma = W_c(\mu,\nu)$. Taking for $(\gamma_\eps)_\eps$ the constant family equal to $\gamma$, we have immediately $H_\eps(\gamma_\eps) = H(\gamma)$ for any $\eps$, such that the upper limit condition is satisfied.
\end{itemize}
As for the $\Gamma$-liminf, given $\gamma_\eps \to \gamma$ in $\Pi(\mu,\nu)$, then,
\begin{itemize}
\item if $\gamma \notin \mathcal{O}_c(\mu,\nu)$, since the entropy is positive on $\Pi(\mu,\nu)$, and $\int c \, d \gamma_{\eps} \to \int c \, d \gamma$, we have that for $\eps $ small enough
$$ H_\eps(\gamma_\eps) \geq \frac{1}{\eps} \left( \int c\dx\gamma_{\eps}-W_c(\mu,\nu)\right) \geq \frac{1}{2\eps} \left( \int c\dx\gamma-W_c(\mu,\nu)\right) \to +\infty; $$
\item otherwise, it is enough to observe, for any $\eps>0$, the inequality $H_\eps(\gamma_\eps) \geq \Ent(\gamma_\eps|\mu\otimes\nu)$ and to conclude by semicontinuity of the entropy. 
\end{itemize}
\end{proof}

The consequences of Theorem~\ref{ZeroGammaCV} and Proposition~\ref{ZeroGammaCV} are the following.
\begin{itemize}
\item The minimal value and the minimizers of $J_\eps$ are converging (up to subsequences) to the minimal transport energy and to an optimal transport plan for cost $c$, respectively. In particular, if there is only one optimal transport plan, then this plan is selected at the limit by the entropic regularization.
\item If the set $\mathcal{O}_c(\mu,\nu)$ contains at least one plan with finite entropy (which is equivalent to saying that $H$ is not identically $+\infty$), then the plan which is selected at the limit is the one having minimal entropy in the set $\mathcal{O}_c(\mu,\nu)$, and we have the asymptotic expansion
$$ \min J_\eps = W_c(\mu,\nu) + \eps \min H + \smallo(\eps) \qquad \text{as }\eps\to 0. $$
\item In the converse case, and if there is no uniqueness of the optimal plan for cost $c$, then the plan which is selected at the limit is unknown, and the only information on the behavior of the minimal value is $(\inf J_\eps - W_c(\mu,\nu))/\eps \to +\infty$.
\end{itemize}

As we will see later on in the paper, this last case may occur when $c(x,y) = |y-x|$, depending on the data $\mu$ and $\nu$; precisely, the next section gives a general description of the optimal plans for distance cost in dimension~one.

\section{Structure of one-dimensional optimal plans}

\subsection{Additional notation and the monotone transport map.} First of all, let us recall the precise definition of the monotone transport map that we will use, which is essentially taken from~\cite[Chapter~2]{SanOTAM}. Let $\mu,\nu$ be two probability measures on $\R$ compactly supported and with no atom, and let us denote by $F_\mu$, $F_\nu$ the cumulative distribution functions of $\mu,\nu$, respectively; we notice that since neither $\mu$ or $\nu$ has any atom, both $F_\mu$ and $F_\nu$ are continuous on $\R$. We define
$$ T(x) := \inf\Big\{ y\in\R \,:\, F_\nu(y) \geq F_\mu(x) \Big\}. $$
We can then prove the following properties, thanks to the boundedness of the supports of $\mu,\nu$ and to the continuity of $F_\mu,F_\nu$:
\begin{itemize}
\item the map $T$ is finitely valued on $(\inf(\supp\mu),+\infty)$;
\item for any $x$ with $\inf(\supp\mu) < x $, the infimum in the definition of $T(x)$ is attained; therefore,
\begin{equation}  \left\{ \begin{array}{l} F_\nu(T(x)) = F_\mu(x)
 \\  \text{and, for any } y<T(x),\; F_\nu(y) < F_\mu(x).
\end{array} \right.  \label{Tfmufnu} \end{equation}
\end{itemize}

The following theorem then holds; see~\cite[Theorems 2.5 and 2.9]{SanOTAM}.

\begin{theorem}The map $T$ is nondecreasing and satisfies $T_\#\mu=\nu$. Moreover, any other such map coincides with~$T$ except on a $\mu$-negligible set. Finally, for any cost function $c$ having form $c(x,y) = h(y-x)$ with $h$ convex, the map $T$ is optimal for the transport problem with cost $c$. \end{theorem}

The following lemma is useful in order to make precise the links between $F_\mu,F_\nu$ and the sign of $T-\id$.

\begin{lemma} \label{LemTFmunu} For $\mu$-a.e.~every point of $\R$, the following equivalences are true:
\begin{enumerate}[label={\normalfont (\alph*)}]
\item $T(x)=x$ if and only if $F_\mu(x) = F_\nu(x)$;
\item $T(x)<x$ if and only if $F_\mu(x) < F_\nu(x)$;
\item $T(x) > x$ if and only if $F_\mu(x)> F_\nu(x)$.
\end{enumerate}
\end{lemma}

\begin{proof} Let us start by the equivalence (b). First of all, let $x$ be a point of the support of $\mu$, different of its bounds, and such that $T(x)<x$; in particular, the monotonicity of $F_\nu$ implies that $F_\nu(T(x)) \leq F_\nu(x)$. Assume moreover that $F_\mu(x) \geq F_\nu(x)$: keeping in mind that $F_\nu\circ T$ and $F_\mu$ coincide, this enforces $F_\nu(x) = F_\nu(T(x))$. We deduce that the interval $(T(x),x)$ does not meet the support of $\nu$: therefore, $T(x)$, which belongs to $\supp\nu$, is actually a boundary point of one of the connected components of $\R\setminus\supp\nu$. These boundary points being countably many, they form a $\nu$-negligible set, and the preimage of this set by $T$ is therefore $\mu$-negligible, proving that the direct implication in~(b) is true for $\mu$-a.e.~$x$.

Conversely, let $x$ be such that $F_\mu(x) < F_\nu(x)$ and assume that $T(x)\geq x$: therefore, since $F_\nu$ is nondecreasing,
$$ F_\nu(T(x)) \geq F_\nu(x) > F_\mu(x), $$
which is impossible since $F_\nu(T(x)) = F_\mu(x)$.

The equivalence (c) can be proved by pretty similar arguments, and (a) is an obvious consequence of (b) and (c).
\end{proof}

In what follows in the paper, we will often use the following sets:
$$ A := \Big\{ x \in \R\,:\, F_\mu(x) = F_\nu(x) \Big\} $$
$$ \text{and} \qquad A^+ := \Big\{ x\in\R \,:\, F_\mu(x) > F_\nu(x) \Big\},\; \qquad A^- := \Big\{ x\in\R \,:\, F_\mu(x) < F_\nu(x) \Big\}. $$
We notice that $A^+$ and $A^-$ are both open and that $A$, which is the complementary set of their union, is closed. Moreover,  thanks to the Lemma~\ref{LemTFmunu}, these three sets coincide, up to $\mu$-negligible sets, with the sets where $T-\id$ is respectively zero, positive and negative. The following remark, though obvious, will also sometimes be useful.

\begin{lemma} \label{lemtrivial} Let $I$ be a maximal interval fully included in $A^-$ or in $A^+$. Then $\mu(I)=\nu(I)$. \end{lemma}

\begin{proof} Such an $I$ is necessarily an open and bounded interval of $\R$; calling it $(a,b)$ and using its maximality and the continuity of $F_\mu$, $F_\nu$, we can easily see that $F_\mu(a)=F_\nu(a)$ and $F_\mu(b) = F_\nu(b)$, so that $\mu(I)=\nu(I)$. \end{proof}

We can now state the precise result on the structure of one-dimensional optimal plans.

\begin{proposition} \label{1DOptPlans} Let $\mu,\nu \in \mathcal{P}(\R)$ be atomless and compactly supported. Then the optimal plans for the Monge problem are exactly the transport plans from $\mu$ to $\nu$ such that, for $\gamma$-a.e.~$(x,y)\in\R^2$,
\begin{enumerate}[label={\normalfont (\alph*)}]
\item if $x \in A$, then $y=x$; \label{impl1}
\item if $x \in A^+$, then $y \geq x$, and $y$ belongs to the same connected component of $A^+$ as $x$;\label{impl2}
\item similarly, if $x \in A^-$, then $y \leq x$, and $y$ belongs to the same connected component of $A^-$ as $x$. \label{impl3}
\end{enumerate} \end{proposition}

The key point of the proof of Proposition~\ref{1DOptPlans} consists in building a suitable Kantorovich potential, which is the topic of the next subsection.

\subsection{Construction of a {K}antorovich potential.}

\begin{lemma} \label{potential} For $x \in \R$, define
$$ u(x) = \int_0^x \big(\ind_{A^+}-\ind_{A^-}\big)(t) \dx t. $$
Then $u$ is a Kantorovich potential from $\mu$ to $\nu$, that is, a maximizer of the problem~\eqref{PbDual}.
\end{lemma}

\begin{proof}[Proof of Lemma~\ref{potential}.] The fact that $u$ is $1$-Lipschitz being obvious, the only non-trivial point of the proof is checking that
\begin{equation} u(T(x))-u(x) = |T(x)-x| \quad \text{for }\mu\text{-a.e. }x\, , \label{egpotentiel} \end{equation}
which will ensure that $u$ is optimal for the dual problem. Let then $x$ be a point of the support of $\mu$ such that $T(x) \neq x$ (otherwise the equality~\eqref{egpotentiel} is obviously satisfied); assume without loss of generality that $T(x) > x$. By contradiction, suppose that $u(T(x))-u(x) < |T(x)-x|$ (the converse inequality cannot hold since $u$ is a $1$-Lipschitz function): from the definition of $u$, this implies that $\mathcal{L}([x,T(x)] \setminus A^+) \neq 0$ and in particular there exists a point $y$ such that $x<y<T(x)$ and $y \notin A^+$, that is, $F_\mu(y) \leq F_\nu(y)$. By the property~\eqref{Tfmufnu}, we deduce~$y\geq T(y)$: summarizing,\
$$ \left\{ \begin{array}{l} x < y ; \\ T(x) > y \geq T(y). \end{array}\right. $$
Since $T$ is a nondecreasing function, this is impossible. The arguments are identical in the case $T(x)<x$.
\end{proof}

The two following paragraphs are then devoted to the rigorous proof of Proposition~\ref{1DOptPlans}.

\subsection{Necessary condition to be optimal.} In this paragraph, we prove that all the optimal plans for the Monge problem satisfy properties~(a)-(c) of Proposition~\ref{1DOptPlans}. We then fix $\gamma \in \O_1(\mu,\nu)$.

{\it \underline{Step I:} for any $\bar x \in A$, we have $\gamma((-\infty,\bar x]\times [\bar x,+\infty)) = \gamma([\bar x,+\infty) \times (-\infty,\bar x]) = 0$.} Let $\bar x \in A$, and assume by contradiction that
\begin{equation} \gamma((-\infty,\bar x]\times[\bar x,+\infty))>0. \label{HalfSpacePositive1} \end{equation}
Notice that \eqref{HalfSpacePositive1} implies that $0 < F_\mu(\bar x) = F_\nu(\bar x) <1$. We first show that
\begin{equation} \gamma([\bar x,+\infty)\times (-\infty,\bar x]) >0. \label{HalfSpacePositive} \end{equation}
Indeed, if \eqref{HalfSpacePositive} was false, then we would have
$$\nu([\bar x,+\infty)) = \mu([x,+\infty))  = \gamma([\bar x,+\infty)\times \R)= \gamma([\bar x,+\infty)^2)   $$
and consequently
$$ \gamma((-\infty,\bar x]\times[\bar x,+\infty)) = \nu([\bar x,+\infty))-\gamma([\bar x,+\infty)\times [\bar x,+\infty)) = 0, $$    
which contradicts \eqref{HalfSpacePositive1}

From \eqref{HalfSpacePositive1} and \eqref{HalfSpacePositive}, we deduce that there exists $(x_1,y_1)$ and $(x_2,y_2)$ in the support of $\gamma$ with
$$ x_1 < \bar x < y_1 \quad\text{and}\quad y_2 < \bar x < x_2. $$
Theorem~\ref{theodual} and the continuity of~$u$ imply that $u(y_1)-u(x_1) = y_1-x_1$ and $u(y_2)-u(x_2) = -(y_2-x_2)$: since $u$ is 1-Lipschitz, this enforces~$u$ to be an affine function with slope~1 and~$-1$ on the whole intervals $[x_1,y_1]$ and $[y_2,x_2]$, respectively. But these intervals have nontrivial intersection, which leads to a contradiction. \smallskip

{\it \underline{Step II:} $\gamma$ satisfies property {\em (a)} of Proposition~\ref{1DOptPlans}.} Let $x \in A$, let $y$ such that $(x,y)$ belongs to the support of $\gamma$, and assume by contradiction that $y>x$. Without loss of generality, we may assume         
$$ (x,y)\cap A \neq \emptyset.$$
Indeed, if this was not the case, $x$ would be a boundary point of one of the connected components of $\R\setminus A$: since this set is open, such points are countably many, forming then a $\mu$-negligible set. Let then $x' \in A$ such that $x < x' < y$. From the result of Step~I, it holds that
\begin{equation} \gamma((-\infty,x']\times[x',+\infty)) = 0. \label{HalfSpaceZero} \end{equation}
In particular, selecting a positive $\eps < \text{min}(x'-x,y-x')$, we deduce from~\eqref{HalfSpaceZero} that $\gamma([x\pm\eps]\times[y\pm\eps])=0$ so that $(x,y)$ does not belong to the support of $\gamma$, a contradiction. Analogously, we prove that the inequality $y<x$ cannot hold except for at most countably points~$x$. \smallskip

{\it \underline{Step III:} $\gamma$ satisfies the  properties {\em (b)} and {\em (c)} of Proposition~\ref{1DOptPlans}.} Let $x \in A^+$ and $y \in \R$ be such that $(x,y)$ belongs to the support of $\gamma$ and $u(y)-u(x) = |y-x|$. First, assume that $y<x$, which implies
$$  -\int_y^x (\ind_{A^+}-\ind_{A^-}) = x-y;$$
in particular, since $A^+$ is an open set and does not meet $A^-$, we must have $A^+ \cap (y,x) = \emptyset$, but since $x$ belongs to $A^+$ which is open, this is impossible.

We then have $y\geq x$; let us now prove that $y$ belongs to the same connected component of $A^+$ as $x$. Since $A^+$ is an open set it suffices to prove that $(x,y) \subset A^+$. As above, we deduce from the equality $u(y)-u(x) = |y-x|$ that $A^-\cap (x,y) = \emptyset$; therefore it suffices to prove that $(x,y)$ does not meet $A$. Assume then that there exists $z\in A$ with $x<z<y$; the result of Step~I gives then
$$ \gamma((-\infty,z]\times[z,+\infty)) = 0.$$
By selecting a positive $\eps< \min(z-x,y-z)$ we obtain again $\gamma([x-\eps,x+\eps]\times[y-\eps, y+\eps]) = 0$, a contradiction since $(x,y)$ belongs to the support of $\gamma$. As before, the arguments in the case $x\in A^-$ are identical.

\subsection{Sufficient condition to be optimal.} \label{suffcond} We now focus on the converse implication of Proposition~\ref{1DOptPlans} and denote by $(I_j^-)_j$, $(I_k^+)_k$ the connected components of $A^-$, $A^+$, respectively (which are at most countably many). We also set for each $j,k$,
$$ \mu_j^- = \mu \res {I_j^-}, \; \mu_k^+ = \mu \res {I_k^+}, $$
$$ \text{and} \quad \nu_j^- = \nu \res {I_j^-}, \; \nu_k^+ = \nu \res{I_k^+}. $$
Lemma~\ref{lemtrivial} implies that $\mu_j^-,\nu_j^-$ (respectively,~$\mu_k^+,\nu_k^+$) have the same mass. Moreover, thanks to the previous subsection, the points (a)-(c) of Proposition~\ref{1DOptPlans} apply for the plan which is induced by the map~$T$, which enforces, for any $j,k$,
$$ \nu_j^- = T_\#\mu_j^- \qquad\text{and}\qquad \nu_k^+ = T_\#\mu_k^+.  $$
Using also Lemma~\ref{LemTFmunu}, we deduce
\begin{align} W_1(\mu,\nu) &= \int |T(x)-x| \dx\mu(x) \nonumber \\
  & = \dsum\limits_j \dint_{I_j^-} (T(x)-x) \dx\mu(x) + \dsum\limits_k \dint_{I_k^+} (x-T(x))\dx\mu(x) \nonumber \\
 & = \dsum\limits_j \left(\dint y\dx\nu_j^-(y) - \dint x\dx\mu_j^-(x)\right) + \dsum\limits_k \left(\dint x \dx\mu_k^+(x) - \dint y \dx\nu_k^+(y)\right). \label{doubleveun} \end{align}
Let now $\gamma$ be a transport plan from $\mu$ to $\nu$ satisfying conditions (a)-(c) of Proposition~\ref{1DOptPlans}. Then we have
\begin{multline*} \int |y-x|\dx\gamma(x,y) = \int_{A\times\R}|y-x|\dx\gamma(x,y) \\ + \sum\limits_j\int_{I_j^-\times\R}|y-x|\dx\gamma(x,y) + \sum\limits_k\int_{I_k^+\times\R} |y-x|\dx\gamma(x,y). \end{multline*}
From the property \ref{impl1}, we immediately deduce that the first integral in the sum above is zero, and from \ref{impl2} and \ref{impl3}, we infer that, for any $j,k$,
$$\int_{I_j^-\times\R}|y-x|\dx\gamma(x,y) = \int_{I_j^-\times I_j^-} (x-y) \dx\gamma(x,y) = \int x \dx\mu_j^- - \int y \dx\nu_j^- $$
$$ \text{and} \qquad \int_{I_k^+\times\R}|y-x|\dx\gamma(x,y) = \int_{I_k^+\times I_k^+} (y-x) \dx\gamma(x,y) = \int y\dx\nu_k^+ - \int x \dx\mu_k^+ .  $$
Adding on all the indexes $j,k$ and using \eqref{doubleveun} leads to $\int|y-x|\dx\gamma(x,y) = W_1(\mu,\nu)$, so that $\gamma$ is optimal, as required.

\section{$\Gamma$-convergence result}

\subsection{Statement of the main theorem} Having at hand the notation of the previous section and of Proposition~\ref{1DOptPlans}, we can now state the main result of this paper.

\begin{theorem} \label{maintheo} Let $\mu,\nu \in \mathcal{P}(\R)$ be two probability measures with bounded support such that $\Ent(\mu | \mathcal{L}^1 ) , \Ent(\nu | \mathcal{L}^1) < \infty$; in particular they are absolutely continuous with respect to the Lebesgue measure. Assume that the following assumptions are true.
\begin{enumerate}[label={\normalfont (H\arabic*)}]
\item The set $A\cap\supp\mu$ has a Lebesgue-negligible boundary, and its interior has the form
$$ \operatorname{int}(A\cap\supp\mu) = \bigcup\limits_i (a_i,b_i) $$
where the intervals $(a_i,b_i)$ are pairwise disjoint and satisfy 
\begin{equation} -\sum\limits_i \int_{a_i}^{b_i} \log \bigl( \min \{x-a_i, b_i-x\} \bigr) \dx  \mu(x) < +\infty \label{ConditionLlogL}. \end{equation}
This assumption is in particular satisfied in the two following cases:
\begin{enumerate}[label={\normalfont (H1\alph*)}]
\item $\mu$ has bounded density on $A$ and $\sum_i |a_i-b_i| \log |a_i-b_i| < + \infty$, or
\item we have $\sum_i |a_i-b_i|^{1-\delta} < + \infty$ for some $\delta>0$.
\end{enumerate}
\item There exists an optimal plan $\gamma$ for the Monge problem such that $\gamma\res(\R\setminus A)^2$ has finite entropy with respect to $\mu\otimes\nu$.
\end{enumerate}
Then the family $(F_\eps)_\eps$ of functionals defined on the set $\Pi(\mu,\nu)$ by
$$ F_\eps(\gamma) = \frac{1}{\eps} \left( \int |y-x|\dx\gamma-W_1(\mu,\nu)\right) + \operatorname{Ent}(\gamma|\mu\otimes\nu)-|\log(2\eps)|\mu(A)  $$
is $\Gamma$-converging, as $\eps\to 0$, to the functional $F$ which is finite only on the set $\O_1(\mu,\nu)$ of optimal plans for the Monge problem, and is then equal to
$$ F(\gamma) =  \Ent(\gamma\res (\R\setminus A)^2 | \mu\otimes\nu) - \int_A \log \left( \frac{\text{\normalfont d}\mu}{\text{\normalfont d}\leb^1} \right)\dx\mu .$$
\end{theorem}

The main qualitative consequence of this result is that the minimizer which is selected by the entropic regularization is exactly the most diffuse one on the zone where an optimal transport plan may have a density, if such a plan having also finite entropy exists (which is the sense of the assumption (H2)). Indeed, Proposition~\ref{1DOptPlans} implies that if the set $A$ has positive mass for $\mu$, all the elements of $\O_1(\mu,\nu)$ have infinite entropy, since their restriction to $A\times A$ is concentrated on the graph on the one-dimensional function $T$. On the contrary, it is natural to guess that there should exist two-dimensional densities on the squares $(I_j^-)^2$, $(I_k^+)^2$ which respect the sign of $T-\id$: selecting the most diffuse such density (that is, the one which minimizes the entropy with respect to $\mu\otimes\mu$ on these squares, provided that such a density exists), we obtain exactly the minimizer of the target functionals $F$. Since, for fixed $\eps>0$, the functional $F_\eps$ is only a rescaling version of the original functional $J_\eps$, it admits exactly the same minimizers, which are therefore converging to the (only) minimizer of $F$.

As we said in the introduction, we may also notice that Theorem~\ref{maintheo} allows us to know precisely the asymptotic behavior of the minimal value of the penalized problem: precisely, as $\eps \to 0$,
$$ \min J_\eps = W_1(\mu,\nu) + \eps|\log(2\eps)|\mu(A) + \eps \min F + o(\eps). $$
The excess of order~$|\log\eps|$ is a common phenomenon with another type of approximation of the Monge problem where the regularization impacts the transport map itself; see~\cite{DepLouSan}. Moreover, this was also the order of convergence of the construction proposed by Carlier~et al.~in~\cite{CarDuvPeySch} (``block approximation'') to prove the density, in the set of transport plans, of transport plans having finite entropy (see Proposition~2.14 therein).
\smallskip

We postpone to the last section an explicit representation formula for the minimizer of the entropy restricted to the squares $(I_j^-)^2$, $(I_k^+)^2$ and in particular a necessary and sufficient condition for the assumption~(H2). Concerning the assumption~(H1), it is essentially needed for technical reasons; we do not claim it to be necessary for the $\Gamma$-convergence result, but unfortunately we were not able to conclude the lower limit estimate without it (see below the proof in subsection~\ref{ParLiminf}).

\subsection{$\Gamma$-limsup inequality}

In this section, we prove the $\Gamma$-limsup inequality of Theorem~\ref{maintheo}, that is, for any fixed $\gamma \in \Pi(\mu,\nu)$, building a family $(\gamma_\eps)_\eps$ in $\Pi(\mu,\nu)$ having $\gamma$ as limit and such that
$$ \limsup\limits_{\eps \to 0} F_\eps(\gamma_\eps) \leq F(\gamma).  $$
We then select $\gamma \in \Pi(\mu,\nu)$ such that $F(\gamma)<+\infty$; otherwise there is nothing to prove. therefore,~$\gamma$ satisfies the conditions \ref{impl1}-\ref{impl3} and the assumption~(H2) of Proposition~\ref{1DOptPlans}. Let $\eps>0$ be fixed. We will define separately two densities whose sum will be our approximating transport plan. \smallskip

Our first transport plan is simply $\gamma_\eps^1 := \gamma\res{(\R\setminus A)^2}$. Thanks to the result of Proposition~\ref{1DOptPlans}, it is exactly the part of $\gamma$ where there is some displacement, so that its Monge cost has value $W_1(\mu,\nu)$. \smallskip

In order to get a transport plan $\gamma_\eps$ from $\mu$ to $\nu$, we have to build a ``complementary transport plan'' $\gamma_\eps^2 \in \Pi(\mu\res A,\mu\res A)$ and to set $\gamma_\eps = \gamma^1_\eps+\gamma^2_\eps$. Let us now see what the requirements are so that the family $(\gamma_\eps)_\eps$ that we will obtain satisfies the $\Gamma$-limsup condition. Assume then that $\gamma_\eps^2 \in \Pi(\mu\res A,\mu\res A)$ is given: then, $\gamma_\eps$ connects $\mu$ to $\nu$ and we have on one hand
\begin{equation} \int |y-x|\dx\gamma_\eps(x,y) = \int|y-x|\dx\gamma_\eps^2(x,y) + W_1(\mu,\nu) \label{mongegammaeps} \end{equation}
and on the other, since $\gamma_\eps^1$ and $\gamma_\eps^2$ have disjoint supports,
\begin{equation} \Ent(\gamma_\eps|\mu\otimes\nu) = \Ent(\gamma_\eps^2|\mu\otimes\nu) + \Ent(\gamma_\eps^1 |\mu\otimes\nu) .\label{entropygammaeps} \end{equation}
Plugging back the equalities \eqref{mongegammaeps} and \eqref{entropygammaeps} into the expression of~$F_\eps$, we can see that the measure~$\gamma_\eps^2$ that we want to build must be a transport plan from $\mu\res A$ to itself which satisfies
\begin{equation} \frac{1}{\eps} \int |y-x|\dx\gamma_\eps^2(x,y) + \Ent(\gamma_\eps^2 | \mu\otimes\nu) \leq \mu(A)|\log(2\eps)| - \int_A \log\frac{\text{d}\mu}{\text{d}\leb^1} \dx\mu + \smallo(1).   \label{BigEstimate} \end{equation}
From now on, we will denote by $\mu_A$ the density of the measure $\mu\res A$; with a slight abuse of notation, we will also use, when required, the same symbols for positive measure on $\R^2$ which are absolutely continuous with respect to the Lebesgue measure and, in that case, their densities (notice that by ``density'' we mean ``density with respect to~$\leb^2$'' and not to $\mu\otimes\nu$). \medskip

{\it \underline{Step I:} the definition of $\gamma_\eps^2$.} The strategy would be easy if we knew exactly the expression of the transport plan which minimizes, for fixed $\eps>0$, the energy $\int|y-x|\dx\gamma + \eps\Ent(\gamma|\mu\otimes\nu)$ on the set $\Pi(\mu_A,\mu_A)$. From the dual formulation of this last problem, it is not hard to see that this plan must have form $a_\eps(x) a_\eps(y) \e^{-|y-x|/\eps}$ for some positive function $a_\eps$; unfortunately, the function $a_\eps$ seems difficult to compute explicitly. Instead, our strategy will be to define directly a first density $\bar\gamma_\eps$ which looks like $a_\eps(x)a_\eps(y)\e^{-|y-x|/\eps}$, and whose marginals are almost $\mu\res A$ and are {\it dominated} by $\mu\res A$; then, we will build another ``complement'' $\tilde\gamma_\eps$ between the remaining part of the marginals.

Precisely, the first density that we define is given by
$$ \bar\gamma_\eps(x,y) := \min(\mu_A(x),\mu_A(y)) \frac{\e^{-|y-x|/\eps}}{2\eps}.$$
We notice that, for $\mu$-a.e.~$x \in A$, we have
\begin{equation} \int_\R \bar\gamma_\eps(x,y)\dx y \leq \mu_A(x) \int_\R \frac{\e^{-|y-x|/\eps}}{2\eps} \dx y = \mu_A(x).  \label{tildemueps} \end{equation}
Let us denote by $\bar\mu_\eps$ the first marginal of $\bar\gamma_\eps$, which coincides with the second one, and let us set $\tilde\mu_\eps := \mu_A-\bar\mu_\eps$: therefore, \eqref{tildemueps} implies that $\tilde\mu_\eps$ is a {\it positive} measure on $\R$. It then remains to build a complementary transport plan $\tilde\gamma_\eps$ belonging to $\Pi(\tilde\mu_\eps,\tilde\mu_\eps)$; then, setting $\gamma_\eps^2 := \bar\gamma_\eps+\tilde\gamma_\eps$ will be enough to get a transport plan from $\mu_A$ to itself.

Our construction of $\tilde\gamma_\eps$ is inspired by the ``block approximation'' in~\cite[Definition~2.9]{CarDuvPeySch} and consists in building a density which is piecewise equal, on small squares, to the multiplication of $\tilde\mu_\eps\otimes\tilde\mu_\eps$ by appropriate constants, keeping in mind that $\tilde\gamma_\eps$ must satisfy a marginal condition. Precisely, we select a family $(Q_i)_i$ of segments recovering the support of $\mu\res A$ (and so of $\tilde\mu_\eps$) and having all length at most $\eps$; notice that thanks to the boundedness of $A$, the number $N_\eps$ of such segments which are needed is bounded by $C/\eps$ for a universal constant $C$. Now we set, for any Borel set $B\subset \R^2$,
$$ \tilde\gamma_\eps(B) = \sum\limits_i \frac{\tilde\mu_\eps\otimes\tilde\mu_\eps(B \cap (Q_i \times Q_i))}{\tilde\mu_\eps(Q_i)}. $$
It is then easy to check that $\tilde\gamma_\eps$ is, as required, a transport plan from $\tilde\mu_\eps$ to itself. In the next steps, we prove that $\bar\gamma_\eps+\tilde\gamma_\eps$ satisfies~\eqref{BigEstimate}. \medskip

{\it \underline{Step II:} the total mass of $\tilde\mu_\eps$ vanishes as $\eps\to 0$.} The result of this step will be proved by a dominated convergence argument: we thus begin by showing that
$$ \text{for } \mu_A\text{-a.e.}~x,\, \tilde\mu_\eps(x) \to 0 \text{ as } \eps\to 0. $$
From the definition of $\bar\mu_\eps$ and $\tilde\mu_\eps$, it is clear that, for $\mu_A$-a.e.~$x$,
$$ 0 \leq \tilde\mu_\eps(x) \leq \int_{y \in \R} |\mu_A(y)-\mu_A(x)| \frac{\e^{-|y-x|/\eps}}{2\eps} \dx y. $$
Denoting by $\alpha_\eps(r) := \frac{r\e^{-r/\eps}}{\eps^2}$, it follows that
\begin{align*} \int_{y \in \R} |\mu_A(y)-\mu_A(x)| \frac{\e^{-|y-x|/\eps}}{2\eps} \dx y & = \int_{y\in \R} \left(\int_{r=|y-x|}^{+\infty} \frac{\alpha_{\eps}(r)}{2r}\dx r\right) |\mu_A(y)-\mu_A(x)| \dx y \\
  & = \int_{r=0}^{+\infty} \left(\frac{1}{2r} \int_{x-r}^{x+r} |\mu_A(y)-\mu_A(x)|\dx y\right) \alpha_\eps(r)\dx r \end{align*}
by the Fubini theorem. Now we notice that, as a measure, $\alpha_\eps$ weakly converges to the Dirac mass~$\delta_0$; on the other hand, we have
$$ \lim\limits_{r \to  0} \left(\frac{1}{2r} \int_{x-r}^{x+r} |\mu_A(y)-\mu_A(x)|\dx y\right) = 0$$
as soon as $x$ is a Lebesgue point of $\mu_A$. This proves that, for $\mu_A$-a.e.~$x\in\R$, the pointwise convergence of densities  $\tilde\mu_\eps(x) \to 0$. The domination assumption can then easily be checked, concluding the proof of this step. \medskip

Before passing to the next steps, let us decompose the energy we are interested in in several terms, namely,
$$ \begin{array}{ll} \dfrac{1}{\eps} \dint |y-x|\dx\gamma_\eps^2(x,y) + \Ent(\gamma_\eps^2|\mu\otimes\nu) &  \\[4mm]
  \hspace{1.7cm} = -2\dint \mu_A \log \mu_A {\dx x} \\[4mm]
  \hspace{2cm} + \dfrac{1}{\eps} \dint |y-x| \dx\bar\gamma_\eps(x,y) + \dint \log(\bar\gamma_\eps) \dx\bar\gamma_\eps & \text{(I)} \\[4mm]
  \hspace{2cm}+ \dfrac{1}{\eps} \dint |y-x|\dx\tilde\gamma_\eps(x,y) & \text{(II)} \\[4mm]
  \hspace{2cm}+ \dint \log\left(1+\frac{\tilde\gamma_\eps}{\bar\gamma_\eps}\right) \dx\bar\gamma_\eps & \text{(III)} \\[4mm]
  \hspace{2cm}+ \dint \log(\tilde\gamma_\eps+\bar\gamma_\eps)\dx\tilde\gamma_\eps & \text{(IV)}.
  \end{array} $$
We will estimate successively the terms (I), (II), (III) and (IV). \medskip

{\it \underline{Step III:} estimates on the terms} (I){\it,} (II){\it, and} (III).  Starting from the definition of $\bar\gamma_\eps(x,y)$, we have for any $(x,y)\in\R^2$
$$ \log\bar\gamma_\eps(x,y) = \log \min(\mu_A(x),\mu_A(y)) -\log(2\eps)-\frac{1}{\eps}|y-x|.$$
Integrating with respect to $\bar\gamma_\eps$ and coming back on the expression of (I) leads to
\begin{equation} \label{FirstEstimI} \text{(I)} = -\log(2\eps) \bar\gamma_\eps(\R^2) + \int \log\min(\mu_A(x),\mu_A(y)) \min(\mu_A(x),\mu_A(y)) \frac{\e^{-|y-x|/\eps}}{2\eps} \dx x \dx y. \end{equation}
Denoting by
$$f_A(x,y) := \min\{\mu_A(x),\mu_A(y)\} \log(\min \{ \mu_A(x),\mu_A(y)\}) $$
$$ \text{and} \qquad g(x)=\mu_A (x) \log ( \mu_A(x)),$$
we claim that
\begin{equation} \label{falogfa} \iint f_A(x,y) \frac{\e^{-|y-x|/\eps}}{2\eps} \dx x \dx y \xrightarrow[\eps\to 0]{} \int_\R g(x)\dx x. \end{equation}
This can be proved exactly in the same way as the result of Step~II; more precisely, we notice that we have $|f_A(x,y)-g(x)| \leq |g(y) - g(x)|$, since $f_A(x,y)$ is either $g(x)$ or $g(y)$. In particular we have
$$ \int_{y \in \R} |f_A(x,y)-g(x)| \frac{\e^{-|y-x|/\eps}}{2\eps} \dx y \leq \int_{r>0} \left( \frac{1}{2r} \int_{x-r}^{x+r} |g(y)-g(x)| \dx y\right) \alpha_\eps(r)\dx r $$
with the same kernel $\alpha_\eps$ as above. It can be checked that the mean value in $r$ goes to zero for a.e.~$x$ thanks to the fact that $g \in L^1$, and since $(\alpha_\eps)_\eps$ has $\delta_0$ as limit in the weak sense of measures, this proves~\eqref{falogfa}. 

By combining~\eqref{FirstEstimI} and~\eqref{falogfa}, we conclude
\begin{equation} \text{(I)} = \bar\gamma_\eps(\R^2)\cdot|\log(2\eps)| + \int \log\mu_A \dx\mu_A + \smallo(1). \label{EstimateI} \end{equation}
\medskip

We now pass to the terms (II) and (III), which are the two easiest terms. First, since $\tilde\gamma_\eps$ is supported in the union of the squares $Q_i\times Q_i$ which have length at most $\eps$, we have $|y-x|\leq \sqrt{2}\eps$ for $\tilde\gamma_\eps$-a.e.~$(x,y)$, so that
$$ 0 \leq \text{(II)} \leq \frac{1}{\eps} \cdot \tilde\gamma_\eps(\R^2) \cdot \sqrt{2}\eps = \sqrt{2} \tilde\mu_\eps(\R) \to 0  $$
since, thanks to the result of Step~II, we have $\tilde\mu_\eps(\R)\to 0$ as $\eps \to 0$.

On the other hand, it also holds by the concavity of the logarithm
$$ 0 \leq \text{(III)} = \int \log\left(1+\frac{\tilde\gamma_\eps}{\bar\gamma_\eps}\right)\dx\bar\gamma_\eps  \leq \int \frac{\tilde\gamma_\eps}{\bar\gamma_\eps} \dx\bar\gamma_\eps  = \tilde\gamma_\eps(\R^2), $$
which is again equal to $\tilde\mu_\eps(\R)$ and therefore vanishes as $\eps \to 0$. The two terms (II) and (III) have therefore zero as limit as $\eps \to 0$. \medskip

{\it \underline{Step IV:} estimate on the term} (IV). This one is the most difficult. Recall that, from the definition of $\bar\gamma_\eps$ and $\tilde\gamma_\eps$, we have for any $(x,y)\in\R^2$
$$ \tilde\gamma_\eps(x,y) = \frac{\tilde\mu_\eps(x)\tilde\mu_\eps(y)}{\tilde\mu_\eps(Q_i)} \leq \frac{\mu_A(x) \mu_A (y)}{\tilde\mu_\eps(Q_i)} $$
$$ \text{and} \quad \bar\gamma_\eps(x,y) = \min(\mu_A(x),\mu_A(y))\frac{\e^{-|y-x|/\eps}}{2\eps} \leq \frac{\mu_A(x)}{2\eps}. $$
It immediately follows that
$$ \text{(IV)} \leq \sum\limits_i \int_{Q_i^2} \log\left(\frac{\mu_A(x) \mu_A(y)}{\tilde\mu_\eps(Q_i)}+\frac{\mu_A(x)}{2\eps}\right)\dx  \tilde{\gamma}_{\eps} (x,y). $$
We separate this sum into two terms, depending if $\tilde{\mu}_\eps(Q_i)$ is larger or smaller than $2\eps$; therefore, we denote by $I_\sm$ the set of indexes $i$ such that $\tilde\mu_\eps(Q_i) \leq 2\eps$, by $I_\la$ its complementary set, and also set
$$ A_\sm := \bigcup\limits_{i\in I_\sm} Q_i \qquad\text{and}\qquad A_\la := \bigcup\limits_{i\in I_\la} Q_i. $$
Denoting by (IV.a) the part of the sum above whose indexes belong to $I_\sm$, we have therefore
\begin{align*} \text{(IV.a)} & \leq \sum\limits_{i \in I_\sm} \int_{Q_i^2} \log\left(\frac{\mu_A(x) \mu_A(y)}{\tilde\mu_\eps(Q_i)}+\frac{\mu_A(x)}{\tilde\mu_\eps(Q_i)}\right)\dx  \tilde{\gamma}_{\eps} (x,y) \\
  & \leq \sum\limits_{i \in I_\sm} \int_{Q_i^2} \bigl(\log (\mu_A(x) ) + \log(\mu_A(y)+1) \bigr)  \dx  \tilde{\gamma}_{\eps} -  \tilde\mu_\eps(Q_i) \log\left(\tilde\mu_\eps(Q_i)\right)  \\
  & \leq  2\int_{A_{\sm}} \log( \mu_A (x) +1) \dx  \tilde{\mu}_{\eps}  -  \sum \limits_{i \in I_\sm} \tilde\mu_\eps(Q_i) \log\left(\tilde\mu_\eps(Q_i)\right).
 \end{align*}
Consequently, since $\log( \mu_A (x) +1)$ is  $\mu_A$ integrable and $\mu_A \geq \tilde{\mu}_{\eps} \to 0$ thanks to the results of Step~I, by dominated convergence we get
\begin{equation} \text{(IV.a)} \leq \smallo(1) - \sum\limits_{i\in I_\sm} \tilde\mu_\eps(Q_i)\log\tilde\mu_\eps(Q_i). \label{IVa} \end{equation}
For the second term, denoting by $N_\eps$ the number of indexes in $I_\sm$, the convexity of $t\mapsto t\log t$ gives
\begin{align*} - \frac{1}{N_\eps} \sum\limits_{i\in I_\sm} \tilde\mu_\eps(Q_i)\log\tilde\mu_\eps(Q_i) & \leq -\left(\frac{1}{N_\eps} \sum\limits_{i\in I_\sm} \tilde\mu_\eps(Q_i)\right) \log\left(\frac{1}{N_\eps} \sum\limits_{i\in I_\sm} \tilde\mu_\eps(Q_i)\right)  \\
  & \leq \frac{1}{N_\eps} \tilde\mu_\eps(A_{\sm}) \left(\log N_\eps - \log\tilde\mu_\eps(A_\sm)\right).
\end{align*}
Plugging this inequality into \eqref{IVa} provides
$$ \text{(IV.a)} \leq \tilde\mu_\eps(A_\sm)\log N_\eps + \smallo(1)  $$
as, at $\eps$ goes to $0$, the term $\tilde\mu_\eps(A_\sm)\log(\tilde\mu_\eps(A_\sm))$ vanishes (thus giving a negligible term once we multiply it by $\eps$). Keeping in mind that $N_\eps \leq C/\eps$ for a constant $C$ which does not depend on~$\eps$ and using again the convergence $\tilde\mu_\eps(A_\sm) \to 0$, we infer
\begin{equation} \text{(IV.a)} \leq |\log\eps| \tilde\mu_\eps(A_\sm) + \smallo(\eps).  \label{IVafinal} \end{equation}
On the other hand, recall that we have $\text{(IV)} \leq \text{(IV.a)}+\text{(IV.b)}$ with

\begin{align*}
\text{(IV.b)} &:= \sum\limits_{i \in I_\la} \int_{Q_i} \log\left(\frac{\mu_A(x) \mu_A(y)}{2\eps}+\frac{\mu_A(x)}{2 \eps}\right)\dx  \tilde{\gamma}_{\eps} (x,y) \\
  & \leq \sum\limits_{i \in I_\la} \int_{Q_i^2} \bigl(\log (\mu_A(x) ) + \log(\mu_A(y)+1) \bigr)  \dx  \tilde{\gamma}_{\eps} -  \tilde\mu_\eps(Q_i) \log\left(2 \eps \right)  \\
  & \leq  \int_{A_{\la}} \log( \mu_A (x) +1) \dx  \tilde{\mu}_{\eps}  +  |\log 2 \eps| \tilde\mu_\eps(A_\la) \\
  & =  |\log 2 \eps| \tilde\mu_\eps(A_\la) + \smallo(1)
  \end{align*}
 using again that the  $\mu_A \geq \tilde\mu_\eps \to 0$ and dominated convergence. Putting together this last estimate with \eqref{IVafinal}, we get
$$ \text{(IV)} \leq |\log(2\eps)| \tilde\mu_\eps(\R) + \smallo(1).  $$

Finally, by adding this last estimate to~\eqref{EstimateI} and using that both~(II) and~(III) have zero limit as $\eps \to 0$, we get
\begin{multline*} \frac{1}{\eps} \int |y-x|\dx\gamma_\eps^2(x,y) + \Ent(\gamma_\eps^2 | \mu\otimes\nu) \\ \leq (\bar\mu_\eps(\R)+\tilde\mu_\eps(\R))|\log(2\eps)| - \int \log\mu_A \dx\mu_A + \smallo(1) \\
   \leq  \mu(A)|\log(2\eps)| - \int \mu_A \log\mu_A\dx x + \smallo(1) \end{multline*}
so that \eqref{BigEstimate} is proven.

\subsection{$\Gamma$-liminf inequality} \label{ParLiminf}

Let $(\gamma_\eps)_\eps$ be a family of transport plans having a limit $\gamma$ as $\eps \to 0$ for the weak convergence of measures. First of all, we start by eliminating the case where $\gamma$ is not an optimal transport plan for the Monge problem by noticing that, in that case,
$$ F_\eps(\gamma_\eps) \geq \frac{M}{\eps}-\mu(A)|\log(2\eps)| $$
for some positive constant $M$; thus the $\Gamma$-liminf inequality is obviously satisfied. We now assume that $\gamma \in \O_1(\mu,\nu)$, so that it verifies the statement of Proposition~\ref{1DOptPlans}; we moreover may assume without loss of generality that all the $\gamma_\eps$, for $\eps>0$ small enough, have finite energy $F_\eps$ and consequently that all of them have a density with respect to the two-dimensional Lebesgue measure. Again we will use generally the same notation for the measures $\gamma_\eps$ and their two-dimensional densities with respect to the Lebesgue measure (and not to $\mu\otimes\nu$).

Let $u$ be the Kantorovich potential given by Lemma~\ref{potential}. Using the $1$-Lipschitz property of $u$, we may write
\begin{align*} \frac{1}{\eps} \left( \int |y-x|\dx\gamma_\eps(x,y) - W_1(\mu,\nu)\right) & = \frac{1}{\eps} \int \Big(|y-x|-(u(y)-u(x))\Big) \dx\gamma_\eps(x,y) \\
  & \geq \frac{1}{\eps} \int_{A \times \R} \Big(|y-x|-(u(y)-u(x))\Big) \dx\gamma_\eps(x,y).\end{align*}

Let us denote by $\gamma^A_\eps  = \gamma_\eps\res A \times \R$ and $\hat\gamma_\eps = \gamma_\eps-\gamma^A_\eps$. Notice that the first marginal of $\gamma^A_{\eps}$ is precisely $\mu_A$, while the second marginal may be different and we call it $\nu_{\eps}$; notice that $\nu_{\eps} \rightharpoonup \nu \res A = \mu_A$ but, thanks to the fact that $\int  \nu \log \nu< \infty$, we still have $\int \nu_{\eps} \log \nu \to \int_A \mu_A \log \mu_A$.

Since $\hat\gamma_\eps$ and $\gamma^A_\eps$ are concentrated on disjoint sets and have $\gamma_\eps$ as sum, we have
$$ \Ent(\gamma_\eps|\mu\otimes\nu) = \Ent(\hat\gamma_\eps|\mu\otimes\nu)+\Ent(\gamma^A_\eps|\mu\otimes\nu). $$
Moreover, it is clear that $\hat\gamma_\eps$ weakly converges, as $\eps\to 0$, to $\gamma\res (\R\setminus A)^2$; by the lower-semicontinuity property of the entropy functional, it follows that
$$ \liminf\limits_{\eps\to 0} \Ent(\hat\gamma_\eps|\mu\otimes\nu) \geq \Ent(\gamma\res (\R\setminus A)^2|\mu\otimes\nu).  $$
Putting together the above estimates, and using the lower semicontinuity of the entropy, it is clear that the $\Gamma$-liminf inequality we look for reduces~to
\begin{multline*} \iint \Big(|y-x|-(u(y)-u(x))\Big) \dx\gamma^A_\eps(x,y) + \eps \Ent(\gamma_\eps^A|\mu\otimes\nu) \\
\geq \eps\left(\mu(A)|\log(2\eps)|-\int_A \log \Bigl( \frac{\text{d}\mu}{\text{d}\leb^1}\Bigr) \dx\mu \right) + o(\eps),\end{multline*}
which, by noticing that $\Ent(\gamma_\eps^A |\mu\otimes\nu) =  \iint \gamma_\eps^A \log\gamma_\eps^A - 2 \int_A \mu_A\log\mu_A + o(1) $, 
is equivalent to
\begin{multline} \iint_{A \times \R} \Big(|y-x|-(u(y)-u(x))\Big) \dx\gamma^A_\eps(x,y) + \eps\iint_{A \times \R} \gamma^A_\eps\log\gamma^A_\eps \\ \geq \eps\left(\mu(A)|\log(2\eps)|+\int_A \mu\log\mu\right) + \smallo(\eps). \label{gammaliminf} \end{multline}
In order to prove \eqref{gammaliminf}, we will make use of the following lemma.

\begin{lemma} \label{jensen} Let $f, c$ be measurable positive functions such that $cf \in L^1(\R)$ and $f \in L\log L(\R)$. Let $\eps > 0$. Then the following inequality holds:
\begin{multline*} \int \big(c(x)f(x) + \eps f(x)\log(f(x))\big)\dx x \geq \eps\left(\left(\int f\right)\log\left(\int f\right) - \left(\int f\right) \log\left(\int \e^{-c/\eps}\right) \right).  \end{multline*} \end{lemma}

\begin{proof}[Proof of Lemma~\ref{jensen}.] We write the left-hand side as
$$ \eps\times  \left(\int \e^{-c/\eps}\right) \times \frac{1}{\int \e^{-c/\eps}} \int \frac{f(x)}{\e^{-c(x)/\eps}} \log\left(\frac{f(x)}{\e^{-c(x)/\eps}}\right) \e^{-c(x)/\eps}\dx x $$
  and the claimed inequality is then a direct consequence of the Jensen inequality, applied to the convex function $t\mapsto t\log t$ and the probability measure $\dfrac{\e^{-c(x)/\eps}\dx x}{\int \e^{-c/\eps}}$. \end{proof}

Let $\eps>0$ and $x\in A$ be fixed, and let $[a,b]$ be a bounded interval containing the supports of $\mu$ and $\nu$. We apply Lemma~\ref{jensen} to the function $f_x: y\mapsto \gamma_\eps^A (x,y)$ and the function $c$ defined as
$$ c_x(y) = \left( |y-x|-(u(y)-u(x)) \right) \ind_{a\leq y\leq b}. $$
We then obtain
\begin{center}
\begin{tabular}{l}
$\dint_a^b \Big(|y-x|-(u(y)-u(x))\Big) \gamma_\eps(x,y) \dx y + \eps \dint_a^b \log\gamma_\eps(x,y) \gamma_\eps(x,y) \dx y$ \\[4mm]
\hspace{0.2cm} $\geq \eps \left(\dint_a^b \gamma_\eps(x,\cdot) \right) \log \left(\dint_a^b \gamma_\eps(x,\cdot) \right) - \eps \left(\dint_a^b \gamma_\eps(x,\cdot) \dx y\right) \log\left( \dint_a^b \e^{-c_x(y)/\eps} \dx y \right)$ \\[4mm]
\hspace{0.2cm}  $\geq \eps \mu_A(x) \log\mu_A(x) -\eps \mu(x) \log\left( \dint_a^b \e^{-c_x(y)/\eps} \dx y \right). \label{test}$
\end{tabular}
\end{center}
After integrating with respect to $x$,  it remains to prove the following lower estimate:
$$ -\int_A \log \left(\int_a^b \e^{-c_x(y)/\eps} \dx y \right) \dx\mu(x) \geq \mu(A) |\log(2\eps)|+\smallo(1). $$
We notice that, from the definition of the Kantorovich potential $u$, it holds that
$$ \Big||y-x|-(u(y)-u(x))\Big| \leq \leb^1(A\cap[x,y]) $$
for any $x,y \in \R$ (and where the notation $[x,y]$ is used alternatively for $[x,y]$ and $[y,x]$, depending if $x\leq y$ or $y \leq x$). Consequently, it is enough to prove the following:
$$ -\int_A \log \left(\int_a^b \e^{-\leb^1([x,y]\cap A)/\eps} \dx y \right) \dx\mu(x) \geq \mu(A) |\log(2\eps)|+\smallo(1), $$
which can be written, after changing of variable $y = x+\eps t$, as
\begin{equation} \liminf\limits_{\eps \to 0} \int_A -\log\left(\frac{1}{2} \int_{I_\eps(x)} \exp\left(- |t|\cdot \frac{\leb^1([x,x+\eps t]\cap A)}{\eps| t|}\right) \dx t\right) \dx\mu(x) \geq 0\, ,  \label{liminffinal} \end{equation}
where we denote by $I_\eps(x)$ the interval $[-(x-a)/\eps,(b-x)/\eps]$.

We prove~\eqref{liminffinal} by a dominated convergence argument. Denote by
$$ \mathcal{I}_\eps(x) := \frac{1}{2} \int_{I_\eps(x)} \exp\left(- |t|\cdot \frac{\leb^1([x,x+\eps t]\cap A)}{\eps| t|}\right) \dx t. $$
Let $x$ be a fixed Lebesgue point of $A$ for $\mu$ and let $\alpha>0$ be arbitrary small. Let $\eta>0$ be fixed such that
$$ 1-\alpha \leq \frac{\leb^1([x,y]\cap A)}{|y-x|} \leq 1+\alpha $$
for any $y\in [x-\eta,x+\eta]$. Setting $y = x+\eps t$, we obtain after multiplying by $t$, composing with the function $\exp(-\cdot)$ and integrating
\begin{multline} \label{Terme1} \frac{1}{1+\alpha} \left(1-\exp\left(-(1+\alpha)\frac{\eta}{\eps}\right)\right) \leq \frac{1}{2} \int_{-\eta/\eps}^{\eta/\eps} \exp\left(- |t|\cdot \frac{\leb^1([x,x+\eps t]\cap A)}{\eps| t|}\right) \dx t \\
\leq \frac{1}{1-\alpha} \left(1-\exp\left(-(1-\alpha)\frac{\eta}{\eps}\right)\right). \end{multline}
The other part of $\mathcal{I}_\eps(x)$ can be estimated simply by noticing that, for any $t$ such that $|t|\geq \eta/\eps$, it holds that
$$ \leb^1([x,x+\eps t]\cap A) \geq \leb^1([x,x+\eta]\cap A) >0, $$
the last inequality coming from the fact that $x$ is a Lebesgue point of $A$. Therefore,
\begin{equation} \frac{1}{2} \int_{I_\eps(x) \setminus [-\eta/\eps,\eta/\eps]} \e^{-\leb^1([x,x+\eps t]\cap A)} \dx t \leq \frac{b-a}{2\eps} \exp\left(-\frac{\leb^1([x,x+\eta]\cap A)}{\eps}\right). \label{Terme2}  \end{equation}
Using \eqref{Terme1} and \eqref{Terme2} and sending $\eps$ to $0$, we obtain
$$ \frac{1}{1+\alpha} \leq \liminf\limits_{\eps \to 0} \mathcal{I}_\eps(x) \leq \limsup\limits_{\eps \to 0} \mathcal{I_\eps}(x) \leq \frac{1}{1-\alpha}$$
and since $\alpha > 0$ is arbitrary, we conclude that $-\log \mathcal{I}_\eps(x) \to 0$ for $\mu$-a.e.~$x\in A$.

To conclude the proof, we have to control $-\log \mathcal I_\eps(x)$ by an $\mu$-integrable function which does not depend on $\eps$; here our assumption~(H1) on the structure of the set $A$ comes into play. Denote~by
$$ l(x) = \min\Big( (x-a_i), (x-b_i)\Big) \quad \text{for } x \in \operatorname{int}(A\cap\supp\mu),\; a_i < x < b_i. $$
Since $A$ has negligible boundary, the function $l$ is well-defined $\mu$-a.e.~on~$A$. Let now $x$ be a point of the interior of $A$ and $i$ such that $a_i< x < b_i$. We claim that
\begin{equation} \text{for any } y \in I, \; \leb^1([x,y]\cap A)\geq \frac{|y-x|}{(b-a)} l(x). \label{alpha1l} \end{equation}
Indeed, this inequality is clearly true if $y \in (a_i,b_i)$, since this implies $\leb^1([x,y]\cap A) = |y-x|$ and since $l(x) < b-a$; on the other hand, if $y > b_i$ for instance, then $\leb^1([x,y]\cap A) \geq (b_i-x) \geq l(x)$, and since $y$ and $x$ both belong to $[a,b]$, yielding $|y-x|/(b-a) \geq 1$.

Inequality~\eqref{alpha1l} directly implies
$$ \mathcal{I}_\eps(x) \leq \frac{1}{2} \int_{I_\eps(x)} \exp\left(-|t|\frac{l(x)}{b-a}\right) \dx t \leq \frac{b-a}{l(x)} $$
and we also notice, by simply using the inequality $\leb^1([x,x+\eps t] \cap A) \leq \eps t$ for any $x$, that $\mathcal{I}_\eps(x)$ is also always at least equal to $1$, for any $x$. Therefore,
$$ 0 \leq \log \mathcal{I}_\eps(x) \leq \log\left(\frac{b-a}{l(x)}\right). $$
The domination assumption is therefore satisfied as soon as $\int_A |\log l(x)|\dx\mu(x) < +\infty$, as it was stated in (H1). 


It remains to prove that (H1) holds as soon as (H1a) or (H1b) does. Let us write, for each $i$, $[a_i,b_i] = [a_i,\alpha_i]+[\alpha_i,\beta_i]+[\beta_i,b_i]$, where
$$ \alpha_i = \min\left(a_i+1,\frac{a_i+b_i}{2}\right) \quad\text{and}\quad \beta_i = \max\left(b_i-1,\frac{a_i+b_i}{2}\right). $$
We then obtain, by controlling $l(x)$ with $b-a$ on each $[\alpha_,\beta_i]$,
\begin{equation} \begin{array}{rcl} \dint_A |\log l(x)|\dx\mu(x) & = & \dsum\limits_i \dint_{a_i}^{b_i} |\log l(x)| \dx\mu(x) \\[4mm]
 & \leq & \dsum_i \left(-\int_{a_i}^{\alpha_i}\log (x-a_i)\dx\mu(x) -  \int_{\beta_i}^{b_i} \log(b_i-x)\dx\mu(x)\right)  \\[4mm]
 & &+ \dsum_i |\log(b-a)|\mu([a_i,b_i]). \nonumber \end{array} \label{logldmu} \end{equation}
The last term in~\eqref{logldmu} is exactly $\mu(A)|\log(b-a)|$, so it remains to control the two first terms. In case assumption (H1a) is true, we simply notice
\begin{align*} 0  & \leq -\int_{a_i}^{\alpha_i}\log (x-a_i)\dx\mu(x) \\
 & \leq - \|\mu\|\infty (\alpha_i-a_i)\Big(\log(\alpha_i-a_i)-1\Big) \\
 & \leq \|\mu\|\infty \frac{b_i-a_i}{2}\left( 1+\left|\log\frac{b_i-a_i}{2}\right| \right) \end{align*}
and the same bound holds for the term $\int_{\beta_i}^{b_i} |\log (b_i-x)|\dx\mu(x)$ by exactly the same computation, leading to
$$ \int_A |\log l(x)|\dx\mu(x) \leq \|\mu\|\infty \leb^1(A\cap\supp\mu) + |\log(b-a)|\mu(A) + \sum\limits_i (b_i-a_i)|\log(b_i-a_i)|, $$
which is finite by assumption (H1a).

Assume now instead that (H1b) holds. We then use the inequality $\delta tu \leq u \log u + e^{\delta t-1}$ and get
\begin{align*}
 \delta \int_{a_i}^{\alpha_i} |\log (x-a_i)| \mu(x) \dx x & \leq \int_{a_i}^{\alpha_i} e^{-\delta\log(x-a_i)} \dx x + \int_{a_i}^{\alpha_i} \log (\mu(x) ) \mu(x) \dx x \\
 & = \int_{a_i}^{\alpha_i} \frac 1{(x-a_i)^{\delta}} \dx x  + \int_{a_i}^{\alpha_i} \log (\mu(x) ) \mu(x) \dx x \\
 & = \frac { (\alpha_i - a_i)^{1-\delta}}{1-\delta} +  \int_{a_i}^{\alpha_i} \log (\mu(x) ) \mu(x) \dx x.
\end{align*}
Using a similar argument for the intervals $\beta_i, b_i$ and then summing up we get
$$  \delta \int_A |\log l(x)|\dx\mu(x)  \leq  2\int_A \log(\mu)\mu \dx x +| \log (b-a)| \mu(A) + \frac 2{1-\delta} \sum_i (b_i-a_i)^{1-\delta},$$
concluding the proof.

\section{An explicit form of the minimizer} \label{LastSection}

In this section we will try to compute explicitly what plan we are selecting via the limit procedure, that is, the minimizer of $F$. Given the expression of $F$ and the result of Proposition~\ref{1DOptPlans} on the structure of~$\mathcal{O}_1(\mu,\nu)$, it is enough to know, on each maximal interval~$I$ where $F_\mu-F_\nu$ has constant sign, the minimizer of the entropy among all the plans $\gamma$ such that $y-x$ has good sign for $\gamma$-a.e.~pairs $(x,y)$. \smallskip

In particular, for each $j$ (and using the notation of subsection~\ref{suffcond}), this plan will be the minimizer, on the set $\Pi(\mu\res I,\nu\res I)$, of the entropy with respect to the measure $k=\mu^-_j  \otimes \nu^-_j  \cdot \ind_{y > x}$. This problem is well known in the literature (see, for example, \cite{Borwein2}), and the shape of the minimizer is known to be of the form $a(x) \otimes b (y) \cdot  k$. The proof of~\cite{Borwein2} relies on an abstract and general result of~\cite{Borwein1}: to be more precise the usual proof relies on the existence of maximizers in the dual problem if the entropy minimization has a finite value; then, using the duality relation, we can get  that the minimizer has the special shape $a(x) \otimes b (y) \cdot  k$. \smallskip

We will provide a self-contained proof of this result, in a more suitable context for our needs: precisely, we prove the existence of a plan with this special shape, independently of the finiteness of the entropy (see Proposition~\ref{prop:existence} below), which we think is an interesting result itself. It would anyway be interesting to know when there exists a plan of the form $a(x) \otimes b (y) \cdot  k$ still when the entropy minimization problem is not finite, besides our case.

In Proposition~\ref{prop:ineqentr} we then prove that the construction in Proposition~\ref{prop:existence} gives in fact the unique minimizer in the optimization problem, whenever its value is finite; here we in fact pass to the dual problem and we use the dual functions given by the particular shape found in Proposition~\ref{prop:existence}. 

Then, thanks to the explicit construction given in Proposition~\ref{prop:existence} we can state in Theorem~\ref{maintheo2} a necessary and sufficient condition in order to satisfy assumption~(H2).

\begin{proposition}\label{prop:existence} Let $ \mu,\nu$ be two probability measures on $\R$, compactly supported and with no atom, and denote by $\mathcal{F} = F_\mu-F_\nu$ (where $F_\mu$, $F_\nu$ are the cumulative distribution functions of $\mu,\nu$, respectively). Let $I=(a,b)$ be a maximal positivity interval for $\mathcal{F}$, that is, $\mathcal{F}(a)=\mathcal{F}(b)=0$ and $\mathcal{F}>0$ on $I$. Then there exist two nonnegative measures $\rho_1,\rho_2$ on $I$ such that, defining
$$G(x) := \rho_2([x,b]) \quad\text{and}\quad F(y):=\rho_1([a,y])$$
for any $x,y \in I$, we have
$$\begin{cases}
G \cdot \rho_1= \mu\\
F \cdot \rho_2= \nu
\end{cases}  \qquad \text{ on $I$.}$$
Moreover $G$ and $F$ are continuous and positive on $(a,b)$, and the positive measure
$$\gamma_0(x,y) := \rho_1(x) \otimes \rho_2(y) \cdot \ind_{y > x}$$
is a transport plan between $\mu\res I$ and $\nu\res I$.

Similarly, if $I$ is a maximal negativity interval of $\mathcal{F}$, there exist two positive measures $\rho_1,\rho_2$ with densities such that
$$ \left\{
\begin{array}{l}G \cdot \rho_1  = \mu, \\ F\cdot \rho_2 = \nu, \end{array} \right.
\qquad \text{where } F(y) = \rho_1([y,b]), \, G(x) = \rho_2([a,x]), $$
the functions $G$ and $F$ being continuous and positive on $(a,b)$, and
$$ \gamma_0(x,y) := \rho_1(x) \otimes \rho_2(y) \cdot \ind_{y < x} $$
belongs to $\Pi(\mu\res I,\nu\res I)$.
\end{proposition}

\begin{proof} We treat only the case where $I$ is a maximal positivity interval of $\mathcal{F}$ and claim that the proof is similar for their negative counterpart. Denote by
$$ G_\mu : x \mapsto \mu([a,x]) = F_\mu(x)-F_\mu(a) $$
and define the function $G_\nu$ in the same way; in particular, the assumption on $\mathcal{F},a,b$ implies that
$$G_\mu > G_\nu \text{ on $I$ \quad and } G_\mu(a) = G_\nu(a) = 0,\; G_\mu(b) = G_\nu(b) = \mu(I).$$
On the other hand, the function $\mathcal{F}$ is continuous and strictly positive on the whole $I$; we can thus define an integral of $\mu / \mathcal{F}$ on $I$ and call it~$T$. We then observe that $\mu/\mathcal{F} \geq \mu/ G_\mu$  so that we have, for any $x_0,x \in I$ and using that $D \log G_\mu =\mu / G_\mu $,
$$T(x_0)- T(x)=\int_x^{x_0} \frac {1}{\mathcal{F}} \dx  \mu \geq \int_x^{x_0}  \frac {1}{G_\mu} \dx  \mu = \log G_\mu(x_0)  -\log G_\mu(x). $$
In particular, since $G_\mu(a)=0$ we have $T(a)=-\infty$. Let us define
$$F:x\in I \mapsto e^{T(x)} \quad\text{and}\quad G:y\in I\mapsto\frac{\mathcal{F}(y)}{F(y)}.$$
The observations above imply $F(a)=0$. Similarly, $G(b)=0$ and $\log(G)$ is a primitive of $-\nu/\mathcal{F}$. Finally, we define $\rho_1=DF$ and $\rho_2=-DG$ as derivative of $BV_{loc}$ functions (since they are increasing and locally bounded); notice by construction $F$ and $G$ are continuous, so $\rho_1$ and $\rho_2$ will have no atoms. Since $T$ and $F$ are $BV_{loc}$ without jump part, the chain rule holds:
$$ \rho_1=DF = D e^{T}= e^T \cdot DT = F \cdot \frac{ \mu}{\mathcal{F} } = \frac {\mu}{G}.$$
In other words, we have $G\cdot \rho_1 = \mu$ and we can prove the equality $ F \cdot \rho_2=\nu$ in the same way. By construction $F$ and $G$ are continuous and nonnegative; moreover, $F \cdot G = \mathcal{F} >0$ on $I$, so that neither $F$ nor $G$ may vanish on $I$.

Define now $\gamma_0$ as in the statement of Proposition~\ref{prop:existence}: it is then clear that, for any $\phi \in C_b(\R)$,
$$\int_{I \times I} \phi(x) \dx  \gamma_0(x,y) =  \int_I \phi(x) G(x) \dx  \rho_1(x) = \int_I \phi(x) \dx  \mu(x)$$
and that, similarly, the second marginal of $\gamma_0$ is $\nu$.
\end{proof}

The next result proves that the transport plan $\gamma_0$ defined in Proposition~\ref{prop:existence} is actually the minimizer of the entropy among the transport plans on $I^2$ satisfying the corresponding sign constraint.

\begin{proposition}\label{prop:ineqentr} Let $\mu,\nu,\rho_1,\rho_2,I$ and $\gamma_0$ as in Proposition~\ref{prop:existence}. Let $\Gamma_0$ be the set of optimal plans from $\mu\res I$ to $\nu\res I$ (that is, the set of $\gamma \in \Pi(\mu\res I,\nu\res I)$ such that $y-x > 0$ for $\gamma$-a.e.~$(x,y)$). Then we have
\begin{equation}\label{eqn:inequalityentro} 
\min_{\gamma \in \Gamma_0} \Ent ( \gamma | \mu \otimes \nu ) = \Ent (\gamma_0  | \mu \otimes \nu),
\end{equation}
this equality also being true if both sides are $+\infty$. Moreover, whenever $\Ent (\gamma_0 | \mu \otimes \nu ) <\infty$, the only minimizer in~\eqref{eqn:inequalityentro} is $\gamma_0$.
\end{proposition}

In the proof of Proposition~\ref{prop:ineqentr} we will make use of the following lemma.

\begin{lemma}\label{lem:convergence} Let $a,b$ two Borel functions defined on an open set $\Omega$ and let us consider a finite measure $\mu$ on $\Omega$. Then, we have that 
$$ \int_{\Omega} (\phi_n(a) + \phi_n(b))_+ \dx  \mu \to \int_{\Omega} (a+b)_+ \dx  \mu, \qquad \int_{\Omega} (\phi_n(a) + \phi_n(b))_- \dx  \mu \to \int_{\Omega} (a+b)_- \dx  \mu, $$
where $\phi_n(t)= \max\{ -n, \min\{ t,n\} \}$. In particular, if either  $(a+b)_-$ or $(a+b)_+$ is in $L^1(\mu)$, we have also
$$\int_{\Omega} (\phi_n(a) + \phi_n(b)) \dx  \mu \to \int_{\Omega} (a+b) \dx  \mu.$$
\end{lemma}

\begin{proof} First we claim that 
\begin{equation}\label{eqn:phonon} (\phi_n(a) + \phi_n(b))_+ \leq (a+b)_+
\end{equation}
(and similarly $(\phi_n(a) + \phi_n(b))_- \leq (a+b)_-$). 

If $a,b >0$ this is obvious, so let us suppose $a>0$ and $b<0$, with $a+b=k >0$ then $a  = k-b$ and so 
$$\phi_n(a) = \phi_n(k-b) \leq \phi_n(k) + \phi_n(-b)= \phi_n(k) - \phi_n(b) \leq k - \phi_n(b), $$
where we used that $\phi_n$ is odd and it is subadditive on the positive numbers. In order to conclude the proof of~\eqref{eqn:phonon} we need to show that $ \phi_n(a) + \phi_n(b) \leq 0 $ whenever $a+b \leq 0$, but this is obvious from the fact that $\phi_n$ is increasing and odd and in particular if $a +b \leq 0$ we have $a \leq -b$ and so
$$\phi_n( a) + \phi_n ( b) \leq \phi_n( -b) + \phi_n (b) =0.$$

In particular, using~\eqref{eqn:phonon} we immediately get 
$$\limsup_{n \to \infty} \int_{\Omega} (\phi_n(a) +\phi_n(b))_+ \dx  \mu \leq  \int_{\Omega} (a+b)_+ \dx \mu,$$
while, since $\lim_{n \to \infty} (\phi_n(a)+\phi_n(b))_+ = (a+b)_+$, by Fatou's lemma we have
$$\liminf_{n \to \infty} \int_{\Omega}  (\phi_n(a) +\phi_n(b))_+\dx \mu \geq \int_{\Omega} (a+b)_+ \dx \mu,$$
thus giving us that $ \int  (\phi_n(a) +\phi_n(b))_+   \to \int (a+b)_+$. We deal with the negative part in the same way. 
 \end{proof}

\begin{proof}[Proof or Proposition~\ref{prop:ineqentr}.] Again we only will deal with an interval~$I$ where $\mathcal{F}$ is positive. Assume that $\Ent(\cdot|\mu\otimes\nu)$ is not identically $+\infty$ on $\Gamma_0$ and let $\gamma \in \Gamma_0$ having finite entropy; with a slight abuse of notation, we will also denote by $\gamma(x,y)$ the density of the plan $\gamma$ with respect to $\mu \otimes \nu$. Let us denote $\Omega=I \times I  \cap \{ y >x\}$;  we know that for every bounded measurable couple of functions $A, B : I \to \mathbb{R}$, we have
$$ \Ent (\gamma  | \mu \otimes \nu)   = \int_{\Omega}  \bigl( \gamma \log (\gamma)  -  (A(x)+B(y) ) \gamma  - \gamma \bigr) \dx  \mu(x) \dx  \nu (y)  + \int_{\Omega} (A+B+1) \dx  \gamma_0. $$
Now we use that $t \log (t) - c t  \geq  -e^{c-1}$ for any real numbers $c,t$; applying this with $c=A(x)+B(y)+1$ and $t= \gamma(x,y)$, we deduce
\begin{equation} \Ent (\gamma  | \mu \otimes \nu) \geq - \int_{\Omega} e^{ A(x)+B(y)} \dx  \mu \dx  \nu + \int_{\Omega} (A(x) + B(y)+1) \dx \gamma_0 (x,y). \label{TwoTerms} \end{equation}
Let now $\phi_n$ be defined as in Lemma \ref{lem:convergence} and define
$$A_n(x) := \phi_n( -\log ( G(y)) ) \quad\text{and}\quad B_n(y)= \phi_n(-\log (F(y))).$$
Equation~\eqref{TwoTerms} holds in particular with $A_n$ and $B_n$ instead of $A$ and $B$, and we will estimate the two terms of the right-hand side from below by $\Ent(\gamma_0|\mu\otimes\nu)$ as $n\to +\infty$. By Lemma~\ref{lem:convergence} we know that $(A_n(x)+B_n(y))_+ \leq ( -\log G(x) - \log F(y))_+$ and so $e^{A_n(x) + B_n(y) } \leq \max\{ 1/G(x)F(y) ,1\}$. Thus, by dominated convergence, we get
$$\int_{\Omega} e^{A_n+B_n} \dx  \mu \dx  \nu \to \int_{\Omega} \frac 1{GF} \dx  \mu \dx  \nu = \int_{\Omega} \dx \gamma_0.$$
As for the second term of~\eqref{TwoTerms}, it is clear that $(-\log G(x) -\log (F(y))_- \in L^1( \gamma_0)$ since
$$ \int_{\Omega} (-\log(G ) -\log (F))_- \dx  \gamma_0 = \int_{\Omega} \log \Bigl(\frac1{GF} \Bigr)_- \frac 1{GF}   \dx \mu \otimes \nu \leq \int_{\Omega} e^{-1} \dx \mu \otimes \nu < \infty.$$
In particular using again Lemma \ref{lem:convergence} we get
$$\int_{\Omega} (A_n(x)+ B_n(y)  +1) \dx  \gamma_0 \to \int_{\Omega} \left(\log  \Bigl( \frac 1{ F(y) G(x)} \Bigr) +1\right) \dx  \gamma_0 .$$
So, putting all the estimates together and using  $\dfrac { \text{d} \gamma_0 }{ \text{d} (\mu \otimes \nu) } = \dfrac 1{ G(x)F(y)}$, we obtain
\begin{align*}
 \Ent (\gamma\res(I \times I) | \mu\otimes\nu) & \geq - \int_{\Omega} e^{ A_n(x)+B_n(y)} \dx  x \dx y + \int (A_n(x) + B_n(y)+1) \dx \gamma_0 (x,y) \\
 & \to - \int_{\Omega} \dx  \gamma_0  + \int_{\Omega} \left(\log  \Bigl( \frac{ \text{d} \gamma_0 }{ \text{d} \mu\otimes \nu } \Bigr)+1\right)  \dx  \gamma_0  = \Ent (\gamma_0  | \mu \otimes \nu),
 \end{align*}
proving that $\gamma_0$ minimizes $\Ent(\cdot|\mu\otimes\nu)$ on $\Gamma_0$. As for the uniqueness it is sufficient to observe that $\Gamma_0$ is a linear space and $\Ent( \cdot | \mu \otimes \nu)$ is strictly convex.
%
%
%
%
\end{proof}

We can now prove the following result, which expresses exactly when the above assumption (H2) holds depending on the cumulative distribution functions $F_\mu,F_\nu$.

\begin{theorem}\label{maintheo2} Let $\mu,\nu$ be two probability measures on $\R$, compactly supported and absolutely continuous with respect to the Lebesgue measure. Let $\mathcal{F} = F_\mu-F_\nu$, where $F_\mu$, $F_\nu$ are the cumulative distribution functions of $\mu,\nu$, and $A^+$, $A^-$ as in Section~3. Then a necessary and sufficient condition for assumption~(H2) is
$$ -\int_{A^+ \cup A^-} \log |\mathcal{F}(x)|\dx\mu(x) < +\infty  $$
and in that case, the minimal entropy among the elements of $\mathcal{O}_1(\mu,\nu)$ is reached by the plan $\gamma$ which coincides with the identity map on $A\times A$ and, on each square $I^2$ where $I$ is a maximal positivity (resp.,~negativity) interval of $\mathcal{F}$, is equal to the $\gamma_0$ defined by Proposition~\ref{prop:ineqentr}. Moreover, for any such interval $I$, we have
\begin{equation}\label{eqn:explicit}\Ent( \gamma_0\res I^2| \mu \otimes \nu) =  -\int_I \log  |\mathcal{F} (x)|  \dx  \mu(x)  - \mu( I).
\end{equation}
\end{theorem}

\begin{proof} Let $I$ be a maximal positivity interval of $\mathcal{F}$ (the set where $\mathcal{F}$ is negative being treated in a similar way). We first notice that, thanks to Proposition \ref{1DOptPlans}(b) we have that $\gamma\res_{I^2}= \gamma\res_{I \times \R}=\gamma\res_{\R \times I}$ and so it is an (optimal) plan between $\mu\res_I$ and $\nu\res_I$. Then we can deduce from Proposition~\ref{prop:ineqentr} that the minimal entropy of $\gamma\res_{I^2}$ among the elements of $\mathcal{O}_1(\mu,\nu)$ is reached by $\gamma_0$, even if it is $+ \infty$; so it is sufficient to prove that equality \eqref{eqn:explicit} holds, even when one of the two sides is not finite. Let us assume, for simplicity, that $I=(0,1)$. Then by definition 
$$ \Ent( \gamma_0 | \mu \otimes \nu )  = - \int_0^1 \int_x^1 \log ( G(x) F(y))  \dx \rho_2 (y) \dx \rho_1 (x).$$
We can split the logarithm in the inner integral since $G(x)$ is constant and use the definition of $F, G$ and the results from Proposition \ref{prop:existence} to get
$$ \Ent( \gamma_0 | \mu \otimes \nu ) = \int_0^1  \left( - \log (G(x)) G(x) +  \int_x^1 \log (F(y)) \dx \rho_2 (y) \right) \dx  \rho_1 (x).$$
Since $-\rho_2 = DG$ we can use integration by parts in order to get
\begin{multline*} \int_x^1 \log ( F(y))\dx  DG(y) \dx y \\ = \Bigl[G(y) \log (F(y)\Bigr]_x^1 - \int_x^1 \frac{G(y) }{F(y)} \dx  DF = - G(x) \log F(x) - \int_x^1 \frac { \text{d} \mu(y)}{F(y)}, \end{multline*}
where we used that $G(y) \log F(y) \to 0$ for $y\to 1$. This is true since for $y$ sufficiently close to $1$ we have $\mathcal{F} (y)<1$ and so in particular $F(1/2) \leq F(y) = \mathcal{F}(y)/G(y) \leq 1/ G(y)$. In particular, using $G(y) \to 0$ as $y \to 1$ we get
$$0 = \lim_{y \to 1} G(y) \log (F(1/2))  \leq  \lim_{y \to 1} G(y) \log (F(y)) \leq  \lim_{y \to 1} - G(y) \log (G(y)) = 0.$$
Now for any nonnegative measure $\eta$ in $(0,1)$ we have that $\int_0^1 f(x) \int_x^1 \frac{ \dx  \eta (y)}{F(y)}  \dx x  =  \eta ( 0,1) $: it is true by integration by parts for any measure compactly supported in $(0,1)$, and then we reach every measure by an approximation argument. We use this to get
\begin{align*} \Ent( \gamma_0 | \mu \otimes \nu) & = \int_0^1 \left( - \log ( G(x)) G(x)  - \log (F(x)) G(x)   - \int_x^1 \frac{ 1}{F(y)} \dx  \mu (y) \right)  \dx  \rho_1 (x)\\
& = - \int_0^1  \log ( \mathcal{F})  \dx   (G \cdot \rho_1) -  \int_0^1 \dx  \mu \\
& = - \int_0^1 \log (\mathcal{F}) \dx  \mu - \mu ( I),
\end{align*}
where we used $F(x) \cdot G(x)= \mathcal{F} (x)$ and in the last passage $G \cdot \rho_1= \mu$, proving the desired equality~\eqref{eqn:explicit}. 
\end{proof}

\thanks{{\bf Acknowledgements.} This work has been mostly developed while the first author was visiting INRIA-Paris, supported by the FSMP Distinguished Professor Program. Both authors wish to thank the anonymous referees for helpful suggestions which allowed to substantially improve several proofs, and in particular for pointing out the references~\cite{Borwein1,Borwein2} and the links with our work.}

\bibliographystyle{acm}
\bibliography{monge1d_revised}

\end{document}